\documentclass[a4paper,10pt]{article}

\title{McShane-type identities for quasifuchsian representations of nonorientable surfaces}
\author{\textsc{Yi Huang}\\
\textit{Yau Mathematical Sciences Center}\\ 
\textit{Tsinghua University}\\
\href{mailto:yihuangmath@tsinghua.edu.cn}{\texttt{yihuangmath@tsinghua.edu.cn}}\\
}
\date{}

\usepackage{amsmath, amsfonts, amssymb, amsthm, amscd, graphicx, float, 
setspace
}

\usepackage[pdfpagemode={UseOutlines},bookmarks=true,bookmarksopen=true, bookmarksopenlevel=0,bookmarksnumbered=true,hypertexnames=false, colorlinks,linkcolor={blue},citecolor={blue},urlcolor={blue}, pdfstartview={FitV},unicode,breaklinks=true,backref=page]{hyperref}

\usepackage{eulervm}

\theoremstyle{plain}
\newtheorem{thm}{Theorem}
\newtheorem*{thm*}{Theorem}
\newtheorem*{repeatthm1}{Proposition~\ref{thm:double}}
\newtheorem*{repeatthm1'}{Proposition~\ref{thm:tqfcharvar}}
\newtheorem*{repeatthm2'}{Theorem~\ref{thm:mcshane}}

\newtheorem{prop}[thm]{Proposition}

\newtheorem{lem}[thm]{Lemma}
\newtheorem{cor}[thm]{Corollary}
\newtheorem{note}{Note}

\theoremstyle{dfn}
\newtheorem{dfn}{Definition}

\theoremstyle{remark}

\newcommand{\abs}[1]{\left\lvert#1\right\rvert}
\newcommand{\norm}[1]{\left\lVert#1\right\rVert}

\begin{document}

\maketitle

\begin{abstract}
We show that Norbury's McShane identity for nonorientable cusped hyperbolic surfaces $N$ generalizes to quasifuchsian representations of $\pi_1(N)$ as well as pseudo-Anosov mapping Klein bottles with singular fibers given by $N$.
\end{abstract}

\section{Introduction}

The influence of Teichm\"uller theory permeates through moduli space theory, complex analysis, complex dynamics, low-dimensional geometry and topology to representation theory, and lies at the confluence of much of modern mathematics. In representation theory, the Teichm\"uller space $\mathcal{T}(S)$ of a finite-area hyperbolic surface $S$ manifests as the character variety for discrete faithful (i.e.: Fuchsian) representations from the surface group $\pi_1(S)$ to $\mathrm{PSL}(2,\mathbb{R})$. Yet another representation theoretic avatar of Teichm\"uller space arises when describing the character variety $\mathcal{QF}(S)$ of characters for quasifuchsian representations of $\pi_1(S)$ into $\mathrm{PSL}(2,\mathbb{C})$. Specifically, Bers's simultaneous uniformization theorem says:
\begin{thm*}[Corollary to Bers's simultaneous uniformization \cite{bersuniform}]
The space $\mathcal{QF}(S)$ of quasifuchsian representations for an oriented surface $S$ is biholomorphic to $\mathcal{T}(S\cup\bar{S})$.
\end{thm*}
\noindent
Here, the space $\mathcal{QF}(S)$ is rendered a complex manifold when regarded as an open subset contained in the  character variety of representations from $\pi_1(S)$ to $\mathrm{PSL}(2,\mathbb{C})$. In comparison, the space $\mathcal{T}(S\cup\bar{S})$ is given the standard complex structure on \cite{MR0124486}, or equivalently, the complex structure from Bers' embedding of Teichm\"uller space as an open domain in the complex vector space of holomorphic quadratic differentials on $S$\cite{MR0111835,MR0130972}.\medskip

Akiyoshi-Miyachi-Sakuma\cite{amsgroups} take advantage of this complex structure and invoke the identity theorem for holomorphic functions to show that McShane's identities\cite{mcshane_thesis, mcshane_allcusps} for cusped hyperbolic surfaces extend to the space of quasifuchsian representations. Given a finite-volume cusped (possibly nonorientable) hyperbolic surface $F$ with a distinguished cusp $p$, let $\mathcal{S}(F)=\mathcal{S}_1(F)\cup\mathcal{S}_2(F)$ denote the union of the following (possibly empty) sets:
\begin{itemize}
\item
let $\mathcal{S}_1(F)$ be the collection of embedded geodesic-bordered (open) $1$-holed  M\"obius bands on $F$ which contain cusp~$p$. We denote an arbitrary $1$-holed M\"obius band $M$ by the unordered pair $\{\alpha_1,\beta_1\}$ of simple closed 1-sided geodesics contained in $M$; and
\item
let $\mathcal{S}_2(F)$ be the collection of embedded geodesic-bordered (open) pairs of pants on $F$ which contain cusp~$p$. We denote an arbitrary pair of pants $P$ in $\mathcal{S}_2(F)$ by the unordered pair $\{\alpha_2,\beta_2\}$ of simple closed 2-sided geodesics on $F$, which, together with cusp~$p$, bound $P$.
\end{itemize}

\begin{note}
We regard cusps as 2-sided geodesics of length $0$, and thus allow $\alpha_2$ or $\beta_2$ to be cusps. And in the special case when $F$ is a 1-cusped torus $S_{1,1}$, the boundary geodesics $\alpha_2$ and $\beta_2$ are both the same curve. 
\end{note}

\begin{note}
\label{note:nonprimitive}
We regard pairs of pants embedded within $1$-holed M\"obius bands $M\in\mathcal{S}_1(F)$ as elements of $\mathcal{S}_2(F)$. In particular, each embedded $1$-holed M\"obius band contains precisely two embedded pairs of pants $P_\alpha,P_\beta\in\mathcal{S}_2(F)$ respectively obtained by cutting $M$ along $\alpha_1$ and $\beta_1$ (see Figure~\ref{fig:mobpants}). To clarify, the pair of pants $P_\alpha$ in Figure~\ref{fig:mobpants} corresponds to $\{\alpha_1^2,\gamma\}$ and the pair of pants $P_\beta$ corresponds to $\{\beta_1^2,\gamma\}$. The choice to use the non-primitive simple closed geodesics $\alpha_1^2$ and $\beta_1^2$ are to uniformize our McShane identities summands, but this is the only context in which our simple closed geodesics are permitted to be non-primitive.
\end{note}

\begin{figure}[h!]
\centering
\includegraphics[width=0.75\textwidth]{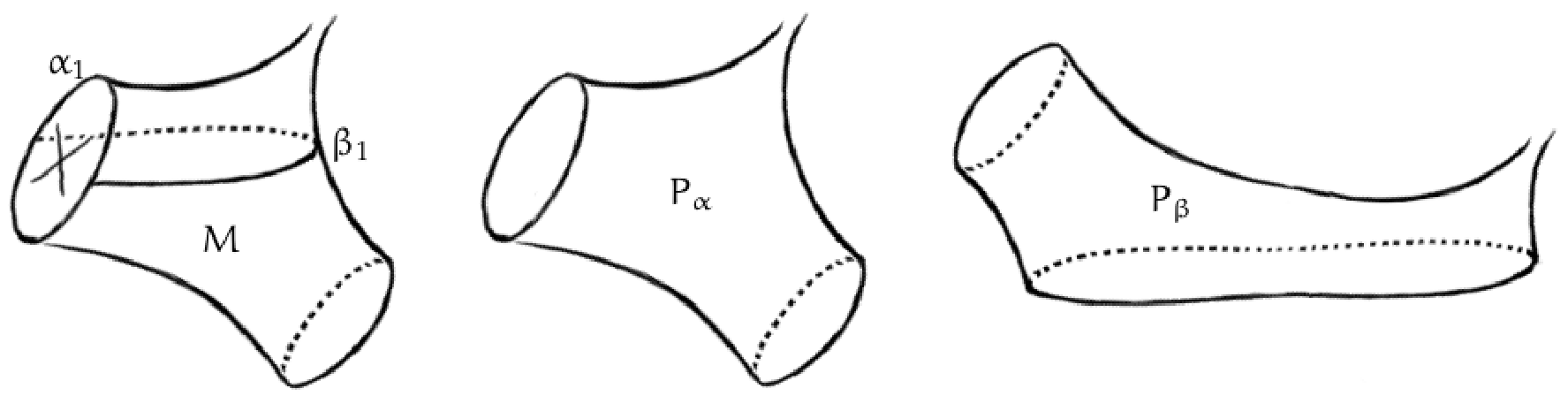}
\caption{(left to right) a $1$-holed M\"obius band $M$; a pair of pants $P_\alpha\subset M$; the other pair of pants $P_\beta\subset M$.}
\label{fig:mobpants}
\end{figure}

\begin{thm*}[Orientable quasifuchsian identity\cite{amsgroups}]
Consider an orientable cusped hyperbolic surface $S$ with a distinguished cusp $p$. For any $\rho\in\mathcal{QF}(S)$, we have the following absolutely convergent series
\begin{align*}
\sum_{\{\alpha_2,\beta_2\}\in\mathcal{S}_2(S)}
\left(e^{\frac{1}{2}(\ell_{\alpha_2}(\rho)+\ell_{\beta_2}(\rho))}+1\right)^{-1}=\frac{1}{2},
\end{align*}  
where $\ell_\alpha(\rho)$ here denotes the complex length (see \S\ref{sec:hyperbolic}) of $\alpha$ taken with respect to $\rho$.
\end{thm*}

\begin{note}
In the special case that $S$ is a 1-cusped torus, the above result is first given in Bowditch\cite{markofftriples}. His strategy of proof is wholly algebraic, employs trace-based cluster algebraic structures corresponding to complexifications of Penner's $\lambda$-lengths\cite{penner} and precedes Akiyoshi-Miyachi-Sakuma's complex analytical approach.
\end{note}

Quasifuchsian surface groups occupy a dense open subset of the set of all Kleinian surface groups\cite{namazisouto,ohshikadensity}. Correspondingly, on the $\mathrm{PSL}(2,\mathbb{C})$ character variety for $\pi_1(S)$, quasifuchsian representations continuously interpolate between holonomy representations for hyperbolic surfaces (i.e.: Fuchsian representations) and restrictions to $\pi_1(S)$ of the holonomy representation for complete finite-volume hyperbolic 3-manifolds. This latter collection of representations constitute boundary points of $\mathcal{QF}(S)$, and parameterize objects such as pseudo-Anosov mapping tori. Bowditch\cite{bowvar} studies this interpolation so as to obtain a McShane-type identity for pseudo-Anosov mapping tori with once-punctured torus fibers, and describes the cusp geometry of these mapping tori in terms of summands of the identity. Akiyoshi-Miyachi-Sakuma generalize Bowditch's work for pseudo-Anosov mapping tori with surface fibers of general type.\medskip

\begin{thm*}[Identity for pseudo-Anosov mapping tori\cite{amsbundles,amsgroups,bowvar}]
Given a pseudo-Anosov map $\varphi:S\to S$, the pseudo-Anosov mapping torus $M_\varphi:=(S\times[0,1])/(x,0)\sim(\varphi(x),1)$ is  which is topologically a fibre-bundle over $\mathbb{S}^1$ with fiber $S$.
Thurston shows that $M_\varphi$ is always a hyperbolic $3$-manifold \cite{thurston}. Denote its holonomy representation by $\phi:\pi_1(M_\varphi)\rightarrow\mathrm{PSL}(2,\mathbb{C})$, and let $\mathcal{S}_\varphi$ denote the collection of unordered pairs $\{\alpha_2,\beta_2\}$ of simple closed geodesics in $M_\varphi$ homotopic to an unordered pair of simple closed geodesics in $\mathcal{S}_2(S)$, then the following series converges absolutely:
\begin{align*}
\sum_{\{\alpha_2,\beta_2\}\in\mathcal{S}_\varphi}
\left(e^{\frac{1}{2}(\ell_{\alpha_2}(\phi)+\ell_{\beta_2}(\phi))}+1\right)^{-1}=0.
\end{align*}  
\end{thm*}

We have, so far, considered the scenario when our reference surface $S$ is orientable. In\cite{norbury}, Norbury considers Fuchsian holonomy representations for nonorientable hyperbolic surfaces via the fact that the group $\mathrm{PGL}(2,\mathbb{R})$ is a double cover of $\mathrm{PSL}(2,\mathbb{R})$ and acts as the group of (potentially orientation-reversing) isometries on $\mathbb{H}^2$ (see, e.g. \S3 of\cite{norbury}). Let $N$ denote a nonorientable hyperbolic surface and $dN$ an orientable double cover, we observe that the fundamental group $\pi_1(dN)$ of the oriented double cover $dN$ of $N$ may be regarded an index $2$ (normal) subgroup of $\pi_1(N)$.

\begin{dfn}[Fuchsian for nonorientable surfaces]
We say that a discrete faithful representation $\rho:\pi_1(N)\rightarrow\mathrm{PGL}(2,\mathbb{R})$ is \emph{Fuchsian} iff. $\rho$ is discrete, faithful and 
\[
\rho(\pi_1(N)-\pi_1(dN))\subset\mathrm{PGL}(2,\mathbb{R})-\mathrm{PSL}(2,\mathbb{R}).
\]
Since $\rho$ is a representation and $\pi_1(dN)$ is an index two subgroup of $\pi_1(N)$, the final condition here is equivalent to requiring that a single homotopy class in $\pi_1(N)-\pi_1(dN)$ be sent to $\mathrm{PGL}(2,\mathbb{R})-\mathrm{PSL}(2,\mathbb{R})$.
\end{dfn}

\begin{thm*}[Norbury's nonorientable cusped surface identity\cite{norbury}]
Given a nonorientable cusped hyperbolic surface $N$ with Fuchsian holonomy representation $\rho$, 
\begin{align*}
\sum_{\{\alpha_1,\beta_1\}\in\mathcal{S}_1(N)}\left(e^{\frac{1}{2}(\ell_{\alpha_1}(\rho)+\ell_{\beta_1}(\rho))}-1\right)^{-1}
+\sum_{\{\alpha_2,\beta_2\}\in\mathcal{S}_2(N)}\left(e^{\frac{1}{2}(\ell_{\alpha_2}(\rho)+\ell_{\beta_2}(\rho))}+1\right)^{-1}
=\frac{1}{2}.
\end{align*}
\end{thm*}

\begin{note}\label{note:1}
Denote the geometric intersection number of two geodesics $\alpha$ and $\beta$ by $\alpha\cdot\beta$. Then, the summand for each of the partial sums in the above identity may be expressed as:
\begin{align*}
\left(e^{\frac{1}{2}(\ell_{\alpha}(\rho)+\ell_{\beta}(\rho))}+(-1)^{\alpha\cdot\beta}\right)^{-1}.
\end{align*}
We henceforth adopt this notational convention for succinctness.
\end{note}

\begin{note}
The above identity is implicitly in\cite{norbury}, which has a McShane identity for bordered nonorientable surfaces. This cusped surface version of the McShane identity may be derived from Theorem~2 of Norbury's paper by a limiting procedure. We give the explicit derivation in Appendix~\ref{sec:appendix} as well as an alternative expression.
\end{note}

There are known extensions of Norbury's identity to quasifuchsian representations when the underlying nonorientable surface is sufficiently topologically simple:
\begin{itemize}
\item
the once-punctured Klein bottle in\cite{norbury}, 
\item
the thrice-punctured projective plane in\cite{markoffquads,hutanzhang} and 
\item
the thrice-bordered projective plane in\cite{malonipalesi}. 
\end{itemize}
In each of these cases, the proof strategy is based on algebraic methods akin to Bowditch's strategy in\cite{markofftriples}.\medskip

\subsection{Main results}
Consider a nonorientable cusped hyperbolic surface $N$ with an oriented double cover $dN$, and let $\iota: dN\rightarrow dN$ denote the involution inducing the quotient map from $dN$ to $N= dN/x\sim\iota(x)$. The primary goal of this paper is to extend Norbury's McShane identities to quasifuchsian representations, much like how Akyoshi-Miyachi-Sakuma generalized McShane's originial identities. To begin with, we need to clarify the notion of a quasifuchsian representation when the underlying surface $N$ is nonorientable. In fact, unlike when the underlying surface $S$ is orientable, there are actually \emph{two} character varieties of representations of $\pi_1(N)$ that we must concern ourselves with, and correspondingly two (homeomorphism) types of hyperbolic $3$-manifolds which we must regard.

\begin{dfn}[Quasifuchsian representation of nonorientable surface groups]
\label{defn:qf}
A \emph{quasifuchsian representations} $\rho:\pi_1(N)\rightarrow\mathrm{PSL}(2,\mathbb{C})$ of a nonorientable surface group $\pi_1(N)$ is a discrete faithful representation whose limit set is quasicircle. 
\end{dfn}

The group $\rho(\pi_1(N))$ is a Kleinian group, and its associated quotient hyperbolic $3$-manifold $\mathbb{H}^3/\rho$ is an orientable hyperbolic $3$-manifold referred to as a \emph{twisted I-bundle}:
\begin{align*}
\left(dN\times\mathbb{R}\right)/(x,t)\sim(\iota(x),-t)=(dN\times [0,\infty))/(x,0)\sim(\iota(x),0).
\end{align*}
Twisted I-bundles are so-named because they are interval bundles over $N$. In particular, they are not product bundles, but resemble a $dN$-fiber bundle over $(-\infty,0]$, albeit with a (non-canonical) singular $N$ fiber over $\{0\}$.\medskip

The other natural geometric generalization of an orientable surface quasifuchsian representation is a nonorientable hyperbolic $3$-manifold homeomorphic to $N\times \mathbb{R}$. Since $\mathrm{PSL}(2,\mathbb{C})$ only allows for orientation-preserving automorphisms on $\mathbb{H}^3$, the holonomy representation for this second class of quasifuchsian representation generalizations instead maps to a $\mathbb{Z}_2$-extension of $\mathrm{PSL}(2,\mathbb{C})=\mathrm{Aut}^+(\mathbb{H}^3)= SO^+(1,3)$ isomorphic to the group $\mathrm{Aut}^\pm(\mathbb{H}^3)= O^+(1,3)$ of \emph{orthochronous Lorentz transformations}. By regarding $\mathrm{PSL}(2,\mathbb{C})= SO^+(1,3)$ as an index $2$ subgroup of $O^+(1,3)$, it is apparent that every quasifuchsian representation may be regarded as a $O^+(1,3)$ representation, we now consider a class of representations which do not arise in such a manner. 

\begin{dfn}[Transflected-quasifuchsian representation of nonorientable surface groups]
\label{defn:tqf}
A \emph{transflected-quasifuchsian representations} $\rho:\pi_1(N)\rightarrow O^+(1,3)$ of a nonorientable surface group $\pi_1(N)$ is a discrete faithful representation whose limit set is quasicircle, such that
\[
\rho(\pi_1(N)-\pi_1(dN))\subset O^+(1,3)-SO^+(1,3).
\]
\end{dfn}

\begin{prop}[Quasifuchsian space]
\label{thm:double}
The space $\mathcal{QF}(N)$ of quasifuchsian representations of $N$, as a holomorphic slice of the $\mathrm{PSL}(2,\mathbb{C})$ representation variety for $\pi_1(N)$, is complex analytically equivalent to the Teichm\"uller space $\mathcal{T}(dN)$. Moreover, the Teichm\"uller space $\mathcal{T}(N)$, regarded as the Fuchsian locus in $\mathcal{QF}(N)$, is a connected and maximal dimensional totally real analytic submanifold of $\mathcal{QF}(N)$.
\end{prop}

\begin{note}
Proposition~\ref{thm:double} actually gives an algebraic approach for describing the complex structure of the Teichm\"uller space of oriented surfaces (as every orientable surface arises as the double cover of some nonorientable surface). We expect the space $\mathcal{QF}(N)$ to be defined by infinitely many algebraic conditions, and hence is not a variety or a quasivariety. Nevertheless, $\mathcal{QF}(N)$ is a very concrete algebraic object with its boundary corresponding to representations of $N$-fibers in various hyperbolic $3$-manifolds. In Corollary~\ref{thm:corroyden}, we (virtually) characterize the group of all biholomorphisms of $\mathcal{QF}(N)$ as the extended mapping class group $\Gamma^\pm(dN)$ on $dN$. It seems possible to write down algebraic expressions for biholomorphisms corresponding to extended mapping classes of $N$ (regarded as elements of $\Gamma^\pm(N)\leq \Gamma^\pm(dN)$), but expressions for general elements are unclear.
\end{note}

Proposition~\ref{thm:double} is a combination of a special case of the main theorem of Bers'\cite{bersspaces} (see Theorem~10.8 of\cite{marden} or Theorem~3.3 of\cite{mcmullen} for a more modern statement, or see Sullivan's work\cite{sullivan} for an even wider-reaching generalization) and general facts to do with antiholomorphic involutions on complex manifolds. Bers' proof of the equivalence of complex structures relies on very general topological arguments; we furnish a slightly different proof in \S\ref{sec:propproof} utilizing cross-ratios, and in-so-doing introducing objects used in the proofs of our main results. The upshot of establishing this complex structure on $\mathcal{QF}(N)$ is to pair it with a version of the identity theorem for multivariate holomorphic functions so that we may prove the following:

\begin{thm}[Identity for quasifuchsian representations of nonorientable surface groups]
\label{thm:mcshane}
Given a nonorientable cusped hyperbolic surface $N$ and a quasifuchsian representation $\rho:\pi_1(N)\to \mathrm{PSL}(2,\mathbb{C})$, define $\mathcal{S}(N)$ to be the set of embedded pairs of pants and $1$-holed M\"{o}bius bands containing cusp $p$ (as per Note~\ref{note:1}). Then, 
\begin{align*}
\sum_{\{\alpha,\beta\}\in\mathcal{S}(N)}\left(e^{\frac{1}{2}(\ell_{\alpha}(\rho)+\ell_{\beta}(\rho))}+(-1)^{\alpha\cdot\beta}\right)^{-1}
=\frac{1}{2}.
\end{align*}
where $\ell_\gamma(\rho)$ is the \emph{complex length} of $\gamma$ (see \S\ref{sec:hyperbolic}).
\end{thm}

Similar results hold for transflected-quasifuchsian representations (see, e.g.: Theorem~3.3 of\cite{mcmullen}):

\begin{prop}[Transflected-quasifuchsian space]
\label{thm:tqfcharvar}
The space $\mathcal{TQF}(N)$ of transflected-quasifuchsian representations of $N$ is real-analytically equivalent to $\mathcal{T}(N)\times\mathcal{T}(N)$.
\end{prop}

\begin{thm}[Identity for transflected-quasifuchsian representations of nonorientable surface groups]
\label{thm:mcshanetransf}
Given a nonorientable cusped hyperbolic surface $N$ and a transflected-quasifuchsian representation $\rho:\pi_1(N)\to O^+(1,3)$, define $\mathcal{S}(N)$ to be the set of embedded pairs of pants and $1$-holed M\"{o}bius bands containing cusp $p$ (as per Note~\ref{note:1}). Then, 
\begin{align*}
\sum_{\{\alpha,\beta\}\in\mathcal{S}(N)}
\mathrm{Re}
\left(e^{\frac{1}{2}(\ell_{\alpha}(\rho)+\ell_{\beta}(\rho))}+(-1)^{\alpha\cdot\beta}\right)^{-1}
=\frac{1}{2}.
\end{align*}
where $\ell_\gamma(\rho)$ is the \emph{complex length} of $\gamma$ (see \S\ref{sec:hyperbolic}).
\end{thm}

\begin{note}
\label{note:length}
Strictly speaking, for transflected-quasifuchsian representations, the length $\ell_\gamma(\rho)$ is only well-defined up to conjugation (\S\ref{sec:hyperbolic}). Thankfully, it is clear that 
\[
\mathrm{Re}
\left(e^z+(-1)^k\right)^{-1}
=
\mathrm{Re}
\left(e^{\bar{z}}+(-1)^k\right)^{-1}
=\tfrac{1}{2}\left(
\left(e^z+(-1)^k\right)^{-1}
+\left(e^{\bar{z}}+(-1)^k\right)^{-1}
\right),
\]
and hence every summand in Theorem~\ref{thm:mcshanetransf} is well-defined. This comment applies also to Theorem~\ref{thm:zero}.
\end{note}

Just as there are two topologically distinct generalizations to the notion of a quasifuchsian representation in the nonorientable setting, we consider two generalizations of pseudo-Anosov mapping tori. The first is geometrically (and topologically) natural: pseudo-Anosov mapping tori with fiber $N$. To begin with, a pseudo-Anosov map for a nonorientable surface is the same as for the orientable setting:

\begin{dfn}[Pseudo-Anosov map for nonorientable surfaces]
Given a nonorientable hyperbolic surface $N$, we call a cusp-fixing homeomorphism $\varphi: N\rightarrow N$ \emph{pseudo-Anosov} if there is a pair $(\mathcal{F}^s,\mathcal{F}^u)$ of measured foliations of $N$ such that:
\begin{itemize}
\item
the stable measured foliation $\mathcal{F}^s$ and the unstable measured foliation $\mathcal{F}^u$ are transverse outside of the singular loci;
\item
the map $\varphi$ preserves the underlying foliations for $\mathcal{F}^s$ and $\mathcal{F}^u$, and acts on $\mathcal{F}^s$ as multiplication by $\kappa^{-1}<1$ and on $\mathcal{F}^u$ as multiplication by $\kappa>1$.
\end{itemize}
See\cite{arnoux, liechtistrenner, pennerpa, strenner} for examples and potential constructions.
\end{dfn}

\begin{note}
To clarify, we only consider cusp-fixing pseudo-Anosov homeomorphisms (i.e.: they do not permute different cusps) in this paper. 
\end{note}

\begin{note}
A homeomorphism $\varphi:N\rightarrow N$ is pseudo-Anosov if and only if it lifts to a pseudo-Anosov map $d\varphi:dN\rightarrow dN$ which commutes with the orientation-reversing involution $\iota:dN\rightarrow dN$. Therefore, by replacing $\varphi$ with $\varphi\circ\iota$ if necessary, we may choose $d\varphi$ to be orientation-preserving.
\end{note}

Given a pseudo-Anosov map $\varphi:N\to N$, its induced pseudo-Anosov mapping torus $M_\varphi$ is again defined as $M_\varphi:=N\times[0,1]/(x,0)\sim(\varphi(x),1)$. Unlike the orientable fiber setting, the pseudo-Anosov mapping torus for a nonorientable fiber $N$ is not an orientable $3$-manifold, and so its holonomy representation $\phi$ does not map to $\mathrm{PSL}(2,\mathbb{C})=SO^+(1,3)$ but to the orthochronous Lorentz group $O^+(1,3)$. The group $\phi(\pi_1(M_\varphi))$ cannot be the limit of quasifuchsian groups but is instead the limit of transflected-quasifuchsian groups $\rho(\pi_1(N))$.

\begin{thm}[Identity for pseudo-Anosov mapping tori with nonorientable fiber]
\label{thm:zeropatori}
Given a pseudo-Anosov map $\varphi: N\rightarrow N$ let $\phi:\pi_1(M_\varphi)\to O^+(1,3)$ denote the holonomy representation for the mapping torus $M_\varphi$. Then,
\begin{align}
\sum_{\{\alpha,\beta\}\in\mathcal{S}_\varphi}
\left(e^{\frac{1}{2}(\ell_{\alpha}(\phi)+\ell_{\beta}(\phi))}+(-1)^{\alpha\cdot\beta}\right)^{-1}=0,
\end{align}
where $S_\varphi=\mathcal{S}(N)/\left(\{\alpha,\beta\}\sim\{\varphi_*\alpha,\varphi_*\beta\}\right)$ denotes the set of homotopy classes in $M_\varphi$ of of pairs of pants (containing cusp $p$) lying on a fiber $N$. See \S\ref{sec:hyperbolic} for clarification on what length should mean in this context.
\end{thm}

The final type of hyperbolic $3$-manifold we consider is orientable, complete, finite-volume and arises as a limit of quasifuchsian (rather than transflected-quasifichsian) groups $\rho(\pi_1(N))$. To begin with, we need a slight twist on the notion of a pseudo-Anosov map:

\begin{dfn}[Twisted-pA pair for nonorientable surfaces]
Given a nonorientable hyperbolic surface $N$, and its orientable double $dN$ such that $N=dN/\iota$. We call a pair of homeomorphisms $(d\varphi,\iota): dN\rightarrow dN$ a \emph{twisted-pseudo-Anosov pair} (twisted-pA pair) if there is a pair $(\mathcal{F}^s,\mathcal{F}^u)$ of measured foliations of $N$ such that:
\begin{itemize}
\item
the stable measured foliation $\mathcal{F}^s$ and the unstable measured foliation $\mathcal{F}^u$ are transverse outside of the singular loci;
\item
the (orientation-reversing) involution map $\iota$ exchanges $\mathcal{F}^u$ and $\mathcal{F}^s$.
\item
the map $d\varphi:dN\to dN$ is a orientation-preserving pseudo-Anosov homeomorphism, takes $\mathcal{F}^u$ to $\kappa\mathcal{F}^u$ and takes $\mathcal{F}^s$ to $\kappa^{-1}\mathcal{F}^s$ for some $\kappa>1$,
\end{itemize}
Note that the foliations for $\mathcal{F}^s$ and $\mathcal{F}^u$ identify to the same transversely self-intersecting ``foliation'' on $N$.
\end{dfn}

\begin{note}
Since $\iota$ exchanges $\mathcal{F}^u$ and $\mathcal{F}^s$, therefore $\iota\circ d\varphi\circ\iota$ takes $\mathcal{F}^u$ to $\kappa^{-1}\mathcal{F}^u$ and $\mathcal{F}^s$ to $\kappa\mathcal{F}^s$. This means that $\iota\circ d\varphi\circ\iota=d\varphi^{-1}$, which in turn asserts that $d\varphi\circ\iota$ is (also) an (orientation-reversing) involution on $dN$.
\end{note}

\begin{dfn}[Pseudo-Anosov mapping Klein bottles]
Given a twisted-pA pair $(d\varphi,\iota): dN\to dN$, we define the $(d\varphi,\iota)$-induced \emph{pseudo-Anosov mapping Klein bottle} $K_{(d\varphi,\iota)}$ as 
\[
K_{(d\varphi,\iota)}:=\left(dN\times\left[0,\tfrac{1}{2}\right]\right)/\left((x,0)\sim(\iota(x),0)\text{ and }\left(x,\tfrac{1}{2}\right)\sim\left(d\varphi\circ\iota(x),\tfrac{1}{2}\right)\right).
\]
The name \emph{mapping Klein bottle} owes to the fact that $K_{(d\varphi,\iota)}$ is a cyclinder $dN\times [0,\frac{1}{2}]$ its ends respectively identified by orientation-reversing involutions $\iota$ and $d\varphi\circ\iota$. The analogous construction for $\mathbb{S}^1\times[0,\frac{1}{2}]$ results in a Klein bottle.
\end{dfn}

\begin{thm}[Identity for pseudo-Anosov mapping Klein bottle]
\label{thm:zero}
Given a twisted-pA pair $(d\varphi,\iota): dN\rightarrow dN$ let $\phi:\pi_1(K_{(d\varphi,\iota)})\to \mathrm{PSL}(2,\mathbb{C})$ denote the holonomy representation for the pA mapping Klein bottle $K_{(d\varphi,\iota)}$. Then,
\begin{align}
\sum_{\{\alpha,\beta\}\in\mathcal{S}_{(d\varphi,\iota)}}
\left(e^{\frac{1}{2}(\ell_{\alpha}(\phi)+\ell_{\beta}(\phi))}+(-1)^{\alpha\cdot\beta}\right)^{-1}=0,
\end{align}
where the set $\mathcal{S}_{(d\varphi,\iota)}:=\mathcal{S}(N)/\sim$ is defined by the equivalence relation generated as follows: $\{\alpha,\beta\}\sim\{\alpha',\beta'\}$ if there are simple closed curves $\hat{\alpha},\hat{\beta},\hat{\alpha}',\hat{\beta}'$ on $dN$ respectively lifting $\alpha,\beta,\alpha',\beta'$ on $N$ such that
\[
\{\hat{\alpha}',\hat{\beta}'\}=\{d\varphi_*\hat{\alpha},d\varphi_*\hat{\beta}\}.
\]
Put simply, $\mathcal{S}_{(d\varphi,\iota)}$ denotes the set of homotopy classes, in $K_{(d\varphi,\iota)}$, of pairs of pants (containing cusp $p$) lying on either of the two twisted $N$-fibers in $K_{(d\varphi,\iota)}$. We again refer to \S\ref{sec:hyperbolic} for clarification on what complex length means in this context.
\end{thm}

In \S\ref{sec:final}, we also show how McShane identities may be used to obtain geometric information about cuspidal tori for pA mapping Klein bottles. This result is significantly more technical in nature and we delay its statement for Theorem~\ref{thm:kleinbottle}.

\section{Character varieties for $\pi_1(N)$}

\subsection{Generalized quasifuchsian representations of nonorientable surface groups}

Let $\rho_0: \pi_1(N)\rightarrow \mathrm{PGL}(2,\mathbb{R})$ be a Fuchsian representation for the nonorientable hyperbolic surface $N$. We regard $\pi_1(dN)$ as a index 2 subgroup of $\pi_1(N)$, and denote the restriction representation to $\pi_1(dN)$ by 
\[
d\rho_0:=\rho_0|_{\pi_1(dN)}: \pi_1(dN)\rightarrow \mathrm{PSL}(2,\mathbb{R})\leq\mathrm{PGL}(2,\mathbb{R}).
\]
Fix an arbitrary $1$-sided curve $\alpha_0\in\pi_1(N)-\pi_1(dN)$ and set $A_0:=\rho_0(\alpha_0)$. Since $\pi_1(dN)$ is an index 2 subgroup of $\pi_1(N)$, it is necessarily a normal subgroup and hence $\pi_1(dN)=\alpha_0\cdot\pi_1(dN)\cdot\alpha_0^{-1}$.\medskip

We concern ourselves with two generalized notions of quasifuchsian representations for $\pi_1(N)$. The philosophy that we take is that:
\begin{enumerate}
\item
a generalized quasifuchsian representation should be a representation $\rho:\pi_1(N)\to\mathrm{Aut}(\mathbb{H}^3)=O^+(1,3)$ into the group of (potentially orientation-reversing) isometries of $\mathbb{H}^3$.
\item
the limit set needs to be a $\rho(\pi_1(N))$-invariant Jordan curve $C_\rho$. In this language, the representations we consider are known as \emph{type I} quasifuchsian representations.
\item
one should be able to deform from $\rho_0$ to any generalized quasifuchsian representation along a continuous path of such generalized quasifuchsian representations.
\end{enumerate}

These are all necessary conditions for quasifuchsian representations, but condition~1 is a strictly weaker condition than the usual requirement that a quasifuchsian representation should map only to orientation-preserving isometries $\mathrm{Aut}^+(\mathbb{H}^3)=\mathrm{PSL}(2,\mathbb{C})$. The usual condition means that every element $\rho_0(\gamma)\in\mathrm{PGL}(2,\mathbb{R})$ extends uniquely to an orientation-preserving isometry of $\mathbb{H}^3$, namely, by regarding $\mathrm{PGL}(2,\mathbb{R})$ as a subgroup of $\mathrm{PGL}(2,\mathbb{C})=\mathrm{PSL}(2,\mathbb{C})$. This approach leads to quasifuchsian representations $\rho:\pi_1(N)\to\mathrm{PSL}(2,\mathbb{C})$ as per Definition~\ref{defn:qf}. We denote the space of characters for quasifuchsian representations $\rho:\pi_1(N)\rightarrow\mathrm{PSL}(2,\mathbb{C})$, regarded as a subset of the character variety 
\[
{\mathrm{Hom}(\pi_1(N),\mathrm{PSL}(2,\mathbb{C}))/\!\!/ \mathrm{PSL}(2,\mathbb{C})},
\]
by $\mathcal{QF}(N)$. Fuchsian representations are a special class of quasifuchsian representations, and we refer to the subset of $\mathcal{QF}(N)$ occupied by Fuchsian representations as the \emph{Fuchsian locus} in $\mathcal{QF}(N)$.\medskip

Utilizing the relaxed form of condition~1 means that each element $\rho_0(\gamma)\in\mathrm{PGL}(2,\mathbb{R})$ has two potential extensions to $\mathrm{Aut}(\mathbb{H}^3)$. For $2$-sided non-peripheral essential curves $\gamma\in\pi_1(N)$, we know that $\rho_0(\gamma)$ is a hyperbolic isometry on $\mathbb{H}^2$ and hence extends either to a hyperbolic isometry (translation) on $\mathbb{H}^3$ or a transflection (glide-plane operation) on $\mathbb{H}^3$. The former is orientation-preserving (i.e.: $\rho_0(\gamma)\in\mathrm{PSL}(2,\mathbb{C})$) whereas the latter is orientation-reversing (i.e.: $\rho_0(\gamma)\in O^+(1,3)-SO^+(1,3)$). Similarly, for $1$-sided essential curves $\gamma\in\pi_1(N)$, which are necessarily non-peripheral, we know that $\rho_0(\gamma)$ is planar transflection (glide-reflection) and extends either orientation-preservingly to a loxodromic isometry (with $\pi$-twist) or to an orientation-reversing transflection of $\mathbb{H}^3$. This added level of flexibility means that there are finitely many distinct representation varieties (and character varieties) of generalized quasifuchsian representations depending upon the choice of orientation-preserving or reversing for the generators of $\pi_1(N)$. Of these many choices, we shall consider the choice which identifies $\mathbb{H}^3/\rho_0(\pi_1(N))$ with $N\times\mathbb{R}$, which is that all $1$-sided curves map to orientation-reversing isometries and all $2$-sided curves map to orientation-preserving isometries. This in turn implies that $\rho_0(\pi_1(dN))$, which is generated by $2$-sided curves, must lie inside $\mathrm{Aut}^+(\mathbb{H}^3)$. Since $\rho_0$ is a representation and $\pi_1(dN)$ is an index $2$ subgroup of $\pi_1(N)$, the existence of the $1$-sided curve $\alpha_0\in\pi_1(N)$ mapping to an orientation-reversing isometry $A_0=\rho_0(\alpha_0)$ would then ensure that
\[
\rho_0(\alpha_0\cdot\pi_1(dN))=\rho_0(\pi_1(N)-\pi_1(dN))\subset\mathrm{Aut}(\mathbb{H}^3)-\mathrm{Aut}^+(\mathbb{H}^3),
\]
hence Definition~\ref{defn:tqf}. We denote the space of characters for transflected-quasifuchsian representations $\rho:\pi_1(N)\rightarrow O^+(1,3)$, regarded as a subset of the character variety 
\[
{\mathrm{Hom}(\pi_1(N),O^+(1,3))/\!\!/ O^+(1,3)},
\]
by $\mathcal{TQF}(N)$. We again refer to the subset of $\mathcal{TQF}(N)$ occupied by Fuchsian representations as the \emph{Fuchsian locus} in $\mathcal{TQF}(N)$.\medskip

All in all, we consider quasifuchsian representations (Definition~\ref{defn:qf}) because they are natural from the perspective of Kleinian group theory, and we focus also on transflected-quasifuchsian representations (Definition~\ref{defn:tqf}) for their topological and geometric naturality in the setting of quasifuchsian hyperbolic $3$-manifold theory.

\subsection{The geometry and topology of generalized quasifuchsian $3$-manifolds}
\label{sec:geotop}

For an orientable surface $S$, the quasifuchsian $3$-manifold $\mathbb{H}^3/\rho(\pi_1(S))$ is homeomorphic to $S\times \mathbb{R}$ (see Theorem~10.2 of\cite{hempel}). This is easily seen when $\rho:\pi_1(S)\rightarrow\mathrm{PGL}(2,\mathbb{R})$ is Fuchsian: the foliation of $\mathbb{H}^3$ into equidistant (non-geodesic) ``planes'' from the central geodesic plane $\mathbb{H}^2\subset\mathbb{H}^3$ is preserved by the action of $\rho$ and hence descends to a foliation of $\mathbb{H}^3/\rho(\pi_1(S))$ into equidistant surfaces surrounding the central copy of $S$. Uhlenbeck\cite{uhlenbeck} showed that $\mathbb{H}^3/\rho(\pi_1(S))$ for almost Fuchsian representations similarly admit a global equidistant foliation from a (unique) central minimal surface. Wang\cite{wang} later showed that any almost Fuchsian $\mathbb{H}^3/\rho(\pi_1(S))$ also admits a unique foliation into constant mean curvature surfaces. These canonical foliations give concrete identifications between the quasifuchsian $3$-manifold $\mathbb{H}^3/\rho(\pi_1(S))$ and $S\times\mathbb{R}$. Note however that these particular foliations do not generalize for arbitrary quasifuchsian representations.\medskip

For a quasifuchsian representation of a nonorientable surface $N$, denote $\rho|_{\pi_1(dN)}$ by $d\rho$ and observe that $d\rho$ is quasifuchsian because $\rho$ and $d\rho$ share the same Jordan curve $C_\rho$ at infinity. Thus Theorem~10.5 of\cite{hempel} ensures that the quasifuchsian $3$-manifold $\mathbb{H}^3/\rho(\pi_1(N))$ is homeomorphic to a twisted interval bundle over $N$, which contains $N$ as a $1$-sided embedded surface. Generally speaking, the $0$-section $N$ of this interval bundle is not canonical. However, when $d\rho$ is an almost-Fuchsian representations, we obtain two different canonical fibrations of $\mathbb{H}^3/\rho(\pi_1(N))$. The first fibration extends Uhlenbeck's result to the nonorientable context and consists of equidistant $dN$ fibers centered around a minimal surface fiber $N$. The second is a family of constant mean curvature $dN$ surrounding the same embedded copy of $N$.\medskip

Similarly, for a transflected-quasifuchsian representation $\rho:\pi_1(N)\to O^+(1,3)$, the restriction representation $d\rho:=\rho|_{\pi_1(dN)}$ may be regarded as a quasifuchsian $SO^+(1,3)=\mathrm{PSL}(2,\mathbb{C})$-representation of $\pi_1(dN)$. In this case, Theorem~10.2 of\cite{hempel} asserts that the quasifuchsian $3$-manifold $\mathbb{H}^3/\rho(\pi_1(N))$ is homeomorphic to the product interval bundle $N\times\mathbb{R}$. Again, this fibration structure is not canonical. Although when $d\rho$ is almost Fuchsian, there are canonical fibrations via either Uhlenbeck's approach or via a family of constant mean curvature $N$.

\subsubsection{Pseudo-Anosov limits of generalized quasifuchsian $3$-manifolds}
\label{sec:pseudo}
We consider two distinct types of ``pseudo-Anosov'' $3$-manifolds in this paper, let us begin with the more familiar: pseudo-Ansov mapping tori $M_\varphi$. Let us study a pseudo-Anosov mapping torus $M_\varphi$ by ``unwrapping'' $M_\varphi$ with respect to its $\varphi$-monodromy to produce $N\times\mathbb{R}$, where $\varphi$ acting by $(x,t)\mapsto (\varphi(x),t+1)$ lifts to a $\mathbb{Z}$-action. We may lift this whole picture to $dN\times\mathbb{R}$, with a lift $d\varphi$ acting by $(x,t)\mapsto (\varphi,t+1)$. In particular, we see that $M_\varphi$ is a $\mathbb{Z}_2$-quotient (fiberwise by $\iota$ since $\iota$ commutes with $d\varphi$) of the mapping torus $M_{d\varphi}$ with orientable fiber $dN$. Since $M_{d\varphi}$ is an orientable $3$-manifold, its monodromy representation is a $\mathrm{PSL}(2,\mathbb{C})$-representation, and we shall make use of this fact for some of our arguments and constructions.\medskip

Similarly, we may unwrap a pseudo-Anosov mapping Klein bottle $K_{(d\varphi,\iota)}$ with respect to the involutions $\iota$ and $\varphi\circ\iota$. This again results in $dN\times\mathbb{R}$ where
\begin{itemize}
\item
$\iota_0:=\iota$ acts by $(x,t)\mapsto (\iota_0(x),-t)$;
\item
$\iota_1:=d\varphi\circ\iota$ acts by $(x,t)\mapsto (\iota_1(x),1-t)$.
\end{itemize}
The above two conditions imply that $d\varphi$ acts by $(x,t)\mapsto (d\varphi(x),t+1)$, and one may generate a $\mathbb{Z}$-family of involutions
\[
\iota_k:=(d\varphi)^k\circ\iota=\iota\circ(d\varphi)^{-k}\text{, which acts on $dN\times\mathbb{R}$ via }(x,t)\mapsto (\iota_k(x),k-t).
\]
Note again that $M_{d\varphi}$ is a $2$-cover of $K_{(d\varphi,\iota)}$, where the quotient is induced by the $\iota$ action on $dN\times\mathbb{R}$ (in particular, the quotient cannot be the fiber-wise action of $\iota$ on each $dN$ fiber). Note also that although the $\iota_k$, regarded as involutions on $dN\times\mathbb{R}$, acts in an orientation-preserving manner.

\subsection{Double uniformization for nonorientable surfaces}\label{sec:propproof}

The aim of this section is to describe the quasifuchsian space $\mathcal{QF}(N)$ and the transflected-quasifuchsian space $\mathcal{TQF}(N)$ for a nonorientable surface group $\pi_1(N)$. We have already seen in Theorem~\ref{thm:tqfcharvar} that:

\begin{repeatthm1'}[Twisted-quasifuchsian space]
The space $\mathcal{TQF}(N)$ of transflected-quasifuchsian representations of $N$ is real-analytically equivalent to $\mathcal{T}(N)\times\mathcal{T}(N)$.
\end{repeatthm1'}

We next consider the quasifuchsian space $\mathcal{QF}(N)$. Fix three arbitrary hyperbolic elements $\gamma_{0},\gamma_{1},\gamma_{\infty}\in\pi_1(dN)$ and normalize every character $[\rho]\in\mathcal{QF}(N)$ to be the representation $\rho$ where the attracting fixed point of $\rho(\gamma_z)$ in $\partial\mathbb{H}^3=\hat{\mathbb{C}}$ is $z$. This is an embedding of the quasifuchsian character variety $\mathcal{QF}(N)$ as a slice within the $\mathrm{PSL}(2,\mathbb{C})$ representation variety for $\pi_1(N)$. In particular, the embedding is algebraic and hence induces a complex structure on $\mathcal{QF}(N)$. We choose to renormalize $\rho_0$ so as to lie on this slice.

\begin{repeatthm1}[Quasifuchsian space]
The space $\mathcal{QF}(N)$ of quasifuchsian representations of $N$, as a holomorphic slice of the $\mathrm{PSL}(2,\mathbb{C})$ representation variety for $\pi_1(N)$, is complex analytically equivalent to the Teichm\"uller space $\mathcal{T}(dN)$. Moreover, the Teichm\"uller space $\mathcal{T}(N)$, regarded as the Fuchsian locus in $\mathcal{QF}(N)$, is a connected and maximal dimensional totally real analytic submanifold of $\mathcal{QF}(N)$.
\end{repeatthm1}

\begin{note}
In specifying the complex structure on $\mathcal{T}(dN)$, we orient $dN$ as the upper-half plane conformal end $\mathbb{H}\subset\hat{\mathbb{C}}/\rho_0(\pi_1(dN))$ rather than the lower-half plane conformal end $\overline{\mathbb{H}}/\rho_0(\pi_1(dN))$.
\end{note}

\begin{proof}
Given an arbitrary quasifuchsian representation $\rho\in\mathcal{QF}(N)$, orient the limit curve $C_\rho$ so that $0,1,\infty\in C_\rho$ are in increasing order, the Jordan domain $\omega_\rho$ bordered counterclockwise by $C_\rho$ gives a marked conformal structure on $dN$ given by the action of $\pi_1(dN)$ on $\omega_\rho$ via $\rho$. This gives a well-defined map
\[
\Phi:\mathcal{QF}(N)\rightarrow \mathcal{T}(dN).
\]
We first show that $\Phi$ is surjective. Given the Beltrami differential $\mu$ corresponding to an arbitrary marked conformal structure in $\mathcal{T}(dN)$, define a new Beltrami differential given by:
\begin{align}\label{eq:beltrami}
\mu_\#(z)
=\left\{
\begin{array}{ll}
\mu(z), &\text{if }z\in\mathbb{H};\\
\mu(A_0\cdot z), &\text{if }z\in\overline{\mathbb{H}};\\
0 & \text{otherwise.}
\end{array}
\right.
\end{align}
Since $||\mu_\#||_\infty<1$, up to M\"obius transformation, there is a unique homeomorphism $\psi_{\mu_\#}:\hat{\mathbb{C}}\rightarrow\hat{\mathbb{C}}$ satisifying the Beltrami equation for $\mu_\#$.\medskip

Consider an arbitrary $\gamma\in\pi_1(N)$, if $\gamma\in\pi_1(dN)$, then $\mu_\#\circ (\rho_0(\gamma))\equiv\mu$
\begin{itemize}
\item
on $\mathbb{H}$ because $\mu$ is $\pi_1(dN)$-invariant;
\item
on $\overline{\mathbb{H}}$ because $A_0\cdot\rho_0(\gamma)\cdot A_0^{-1}$ is in $\rho_0(\pi_1(dN))$.
\end{itemize}
Similarly, if $\gamma\in\pi_1(N)-\pi_1(dN)=\alpha_0^{-1}\cdot\pi_1(dN)$, then $\mu_\#\circ (\rho_0(\gamma))\equiv\mu$ 
\begin{itemize}
\item
on $\mathbb{H}$ because $A_0\cdot\rho_0(\gamma)$ is in $\rho_0(\pi_1(dN))$;
\item
on $\overline{\mathbb{H}}$ because $\rho_0(\gamma)\cdot A_0^{-1}$ is in $\rho_0(\pi_1(dN))$.
\end{itemize}
Thus, for any $\gamma\in\pi_1(N)$, the maps $\psi_{\mu_\#}$ and $\psi_{\mu_\#}\circ\rho_0(\gamma)$ both satisfy the Beltrami equation. The uniqueness of solutions to the differential equation, up to M\"obius transformation, tells us that there is an element $A_\gamma\in\mathrm{PSL}(2,\mathbb{C})$ such that 
\begin{align}
A_\gamma\circ\psi_{\mu_\#}\equiv\psi_{\mu_\#}\circ\rho_0(\gamma).\label{equivariance}
\end{align}
Define a map $\rho_\mu:\pi_1(N)\rightarrow\mathrm{PSL}(2,\mathbb{C})$ that takes $\gamma$ to $A_\gamma$. The fact that this is a representation is due to \eqref{equivariance}. Since $\psi_{\mu_\#}$ is a homeomorphism, we see that $\rho$ is a quasifuchsian representation and hence $\Phi$ is surjective.\medskip

To see that $\Phi$ is injective, consider (equivalently normalized) quasifuchsian representations $\rho_1,\rho_2$ such that $\Phi(\rho_1)\equiv\Phi(\rho_2)$. By Bers' original arguments, the two respective conformal ends of the quasifuchsian representations $d\rho_1$ and $d\rho_2$ are equivalent, and hence $d\rho_1\equiv d\rho_2$ as representations and the Jordan curves $C_{\rho_1}, C_{\rho_2}$ are equivalent. This in turn means that the attracting and repelling fixed points of $\rho_1(\alpha_0)$ and $\rho_2(\alpha_0)$ must be the same. Moreover, since $\alpha_0^2\in\pi_1(dN)$, the real part of the translation lengths for $\rho_1(\alpha_0)$ and $\rho_2(\alpha_0)$ must agree and their imaginary components are equivalent up to addition by either $0$ or $i\pi$. Howover, we know that these two transformations exchange the two components of $\hat{\mathbb{C}}-C_{\rho_1}$ and this ensures that $\rho_1(\alpha_0)=\rho_2(\alpha_0)$. Therefore, the representations $\rho_1$ and $\rho_2$ are equivalent on $\alpha_0\cdot\pi_1(dN)$ and hence on all of $\pi_1(N)$.\medskip

We next show that $\Phi^{-1}$ is a holomorphic map. Putting this with the bijectivity of $\Phi^{-1}$ and Hartog's theorem ensures the biholomorphicity of $\Phi$. Let $\mu^t(\cdot)$ be a complex analytic family of Beltrami differentials in $\mathcal{T}(dN)$ around $\mu=\mu^0$. By examining equation~\eqref{eq:beltrami}, we see that $\mu^t_\#$ is also a complex analytic family of Beltrami differentials. Then, by the holomorphic dependence of the family $\{\psi^t:=\psi_{\mu^t_\#}\}$ of quasiconformal mappings (see, for example, the immediate Corollary to Theorem~4.37 of\cite{imayoshitaniguchi}), we know that for any $z\in\hat{\mathbb{C}}$, the point $\psi^t(z)\in\hat{\mathbb{C}}$ varies holomorphically with respect to $t\in\mathbb{C}$. We also know from the bijectivity of $\Phi$ that
\begin{align}
\Phi^{-1}(\mu^t)=\rho_{\mu^t}=\psi_{\mu^t_\#}\circ\rho_0\circ\psi_{\mu^t_\#}^{-1}.
\end{align}
To show that $\Phi^{-1}$ is holomorphic, it suffices to show that $\rho_{\mu^t}$ varies holomorphically with respect to $t$. Now, given any non-peripheral element $\gamma\in\pi_1(N)$, the cross-ratio
\begin{align}
\left(\rho_{\mu^t}(\gamma)^+,\rho_{\mu^t}(\gamma)^-;z,\rho_{\mu^t}(\gamma)\cdot z\right)\text{, of }
\end{align}
\begin{itemize}
\item
the attracting fixed point $\rho_{\mu^t}(\gamma)^+$ of $\rho_{\mu^t}(\gamma)$,
\item
the repelling fixed point $\rho_{\mu^t}(\gamma)^-$ of $\rho_{\mu^t}(\gamma)$, 
\item
an arbitrary point $z$ away from $\rho_{\mu^t}(\gamma)^\pm$ and 
\item
its image $\rho_{\mu^t}(\gamma)\cdot z$ under the action of $\rho_{\mu^t}(\gamma)$,
\end{itemize}
varies holomorphically with respect to $t$. This cross-ratio suffices to recover the trace of $\rho_{\mu^t}(\gamma)$ up to sign, and since $\mathcal{T}(dN)$ is a simply connected domain, we may choose the correct sign for the trace by making the desired choice on the Fuchsian locus and analytically continuing over the entire character variety. By Hartog's theorem, the composition of $\Phi^{-1}$ and any trace function $\mathrm{tr}\circ\rho(\gamma)$ (for non-peripheral $\gamma$) is a holomorphic function on $\mathcal{T}(dN)$, and since trace functions give global coordinates on the character variety $\mathcal{QF}(N)$, we obtain the desired holomorphicity of $\Phi^{-1}$ and hence the agreement of complex analytic structure on $\mathcal{QF}(N)$ and $\mathcal{T}(dN)$.\medskip

Finally, we show that the Fuchsian locus $\mathcal{T}(N)\subset\mathcal{QF}(N)=\mathcal{T}(dN)$ is a maximal dimensional totally real analytic submanifold. To clarify, we need to show that $\mathcal{T}(N)$ is half-dimensional and that for every point $x\in\mathcal{T}(N)$, we have \[
T_x\mathcal{T}(N)\cap J\left(T_x\mathcal{T}(N)\right)=\{0\}, 
\]
where $J$ denotes the almost complex structure on $\mathcal{QF}(N)$ (see, for example, Definition~5.2 of\cite{loustau}). To show this, we consider the antiholomorphic involution $\iota$ on $\mathcal{T}(dN)$ given by flipping the underlying orientation of $dN$. This action, when interpreted as an action on $\mathcal{QF}(N)=\mathcal{T}(dN)$, is equivalent to precomposing a given Beltrami differential $\mu\in\mathcal{QF}(N)$ by the complex conjugation map on $\hat{\mathbb{C}}$. The fixed-point locus of $\iota$ is precisely the Fuchsian locus $\mathcal{T}(N)$. By a general characterization of maximal totally real analytic submanifolds (see, for example, Prop~6.3 of\cite{loustau}), we conclude that $\mathcal{T}(N)$ is a connected half-dimensional totally real analytic submanifold of $\mathcal{QF}(N)$. 
\end{proof}

\begin{note}
By combining Proposition~\ref{thm:double} with the classical quasifuchsian character variety obtained from Bers' simultaneous uniformization theorem, we see that Proposition~\ref{thm:double} holds true even after replacing $N$ with a (possibly disconnected) complete finite-area hyperbolic surface $F$ and $dN$ with an oriented double cover $dF$ of $F$. 
\end{note}

\begin{cor}[Characterization of biholomorphisms of $\mathcal{QF}(N)$]
\label{thm:corroyden}
The (orientation-preserving) mapping class group $\Gamma^+(dN)$ of the oriented double cover $dN$ is the group of biholomorphisms of $\mathcal{QF}(N)$; except when $N$ is Dyck's surface (the sphere with three cross-caps), in which case the automorphism group is $\Gamma^+(dN)$ modulo the $\mathbb{Z}_2$ generated by the hyperelliptic involution on $dN$.
\end{cor}

\begin{proof}
This is an immediate consequence of Proposition~\ref{thm:double} and Royden's theorem, which asserts that the automorphism group of $\mathcal{T}(dN)$ is the mapping class group $\Gamma^+(dN)$; except when $dN$ is the genus $2$ oriented closed surface (and hence $N$ is Dyck's surface), in which case we need to take $\Gamma^+(dN)$ modulo the hyperelliptic involution.
\end{proof}

\begin{cor}
Elements within the mapping class group $\Gamma^\pm(N)$ act on $\mathcal{QF}(N)$ either biholomorphically or anti-biholomorphically. In particular, the index $2$ (normal) subgroup $\Gamma^+(N)$ of $\Gamma^\pm(N)$ which acts biholomorphically on $\mathcal{QF}(N)$ is also known as the \emph{twist group} -- the subgroup generated by Dehn-twists along $2$-sided curves.
\end{cor}

\begin{proof}
Homeomorphisms on $N$ lift to homeomorphisms on $dN$ and this embeds the mapping class group $\Gamma^\pm(N)$ as a subgroup of the (possibly orientation-reversing) mapping class group $\Gamma^\pm(dN)$ of the oriented double cover $dN$. First note that the action of $\Gamma^\pm(N)$ on $\mathcal{T}(dN)$, regarded as a subgroup of $\Gamma^\pm(dN)$, is precisely the action of $\Gamma^\pm(N)$ on $\mathcal{QF}(N)=\mathcal{T}(dN)$. This is easy to see on the Fuchsian locus, and hence holds true in general because of the topological nature of this action. Since $\Gamma^+(dN)$ acts holomorphically on $\mathcal{T}(dN)$ and $\Gamma^-(dn)=\Gamma^\pm(dN)-\Gamma^+(dN)$ acts antiholomorphically, we obtain the holomorphic/antiholomorphic nature of the action of $\Gamma^\pm(N)$.\medskip

Next note that cross-cap slides (see\cite{szepietowski}) lift to orientation-reversing mapping classes, and so $\Gamma^\pm(N)$ does not embed as a subgroup of $\Gamma^+(dN)$. In particular, this means that the subgroup $\Gamma^\pm(N)\cap\Gamma^+(dN)$ of holomorphically acting mapping classes has index at least $2$ in $\Gamma^\pm(N)$. However, the twist group $\Gamma^+(N)$ is a subgroup of $\Gamma^\pm(N)\cap\Gamma^+(dN)$ because Dehn twists along $2$-sided curves lift to orientation-preserving mapping classes. Since $\Gamma^+(N)$ has index $2$ in $\Gamma^\pm(N)$, we conclude that the twist group $\Gamma^+(N)$ \emph{is} the holomorphic subgroup $\Gamma^\pm(N)\cap\Gamma^+(dN)$.
\end{proof}

\subsection{Complex lengths}
\label{sec:hyperbolic}

\subsubsection{For quasifuchsian representations}
Theorem~\ref{thm:mcshane} of this paper regarding quasifuchsian representations are stated in terms of the complex lengths $\ell_\gamma$ of curves $\gamma$. Given a quasifuchsian representation $\rho$, the real component of the \emph{complex length} $\ell_\gamma(\rho)$ of a curve $\gamma$ is defined as the translation length 
\begin{align}
\mathrm{Re}(\ell_{\gamma}(\rho)):=\inf_{x\in\mathbb{H}^3} d_{\mathbb{H}^3}(x,\rho(\gamma)\cdot x)\label{eq:translationlength}
\end{align}
of $\rho(\gamma)$. For non-peripheral $\gamma$, this is equivalent to the length of the unique geodesic representative of $\gamma$ in $\mathbb{H}^3/\rho(\pi_1(F))$. When $\rho(\gamma)$ is loxodromic (including hyperbolic), the imaginary component $\mathrm{Im}(\ell_{\gamma}(\rho))$ of the complex length $\ell_\gamma(\rho)$ is defined in terms of the rotation angle $\theta+2\pi\mathbb{Z}\in (\mathbb{R}/2\pi\mathbb{Z})$ of the loxodromic transformation $\rho(\gamma)$ around its invariant axis. If $\gamma$ is a $2$-sided curve, then $\mathrm{Im}(\ell_\gamma(\rho)):=\theta+2\pi\mathbb{Z}$. If $\gamma$ is a $1$-sided curve, then $\mathrm{Im}(\ell_\gamma(\rho)):=\theta-\pi+2\pi\mathbb{Z}$. This normalization for the complex length of $1$-sided geodesics $\gamma$ by subtracting $i\pi$ yields the unique holomorphic function $\ell_\gamma:\mathcal{QF}(N)\rightarrow\mathbb{C}/2i\pi\mathbb{Z}$ that agrees with the translation length of $\gamma$ on the Fuchsian locus. And if $\gamma$ is parabolic (this arises when $\gamma$ is peripheral) we set its imaginary component to be $0$, and hence its total complex length is $0$.\medskip

We have defined complex geodesic length $\ell_\gamma(\rho)$ to be functions from $\mathcal{QF}(N)$ to $\mathbb{C}/2\pi i\mathbb{Z}$, but this is insufficient for our purposes, as we always exponentiate half of these lengths in our identities, thereby leading to an ambiguity of sign in the summands. Thankfully, as we are dealing with quasifuchsian representations, it is possible to invoke the simply-connectedness of $\mathcal{QF}(N)$ (Proposition~\ref{thm:double}) to lift these length functions to maps of the form $\ell_\gamma:\mathcal{QF}(N)\rightarrow\mathbb{C}$ via analytic continuation, such that $\ell_\gamma$ is equal to the translation length of $\gamma$ on the Fuchsian locus.\medskip

We now provide a more algebraic formulation of the complex length function for quasifuchsian representations. Fix a lift of $\rho_0$ to a $\mathrm{GL}(2,\mathbb{C})$ representation $\hat{\rho}_0:\pi_1(N)\rightarrow\mathrm{GL}(2,\mathbb{C})$ so that $2$-sided curves have determinant $1$ and $1$-sided curves have determinant $-1$ and use the simply connectedness of $\mathcal{QF}(N)$ to continuously extend this lift over all of $\mathcal{QF}(N)$. Having done so, we may define complex length as follows:

\begin{dfn}[Complex length for quasifuchsian representations]
\label{dfn:qfcomplexlength}
When $\gamma$ is 1-sided, its complex length is defined to be $2\mathrm{arcsinh}\left(\left|\frac{1}{2}\mathrm{tr}\circ\hat{\rho}(\gamma)\right|\right)$ of the trace of $\hat{\rho}(\gamma)$ along the Fuchsian locus, and the analytically extension of this function elsewhere on $\mathcal{QF}(N)$; when $\gamma$ is a $2$-sided geodesic, its complex length $\ell_\gamma$ is defined to be $2\mathrm{arccosh}\left(\left|\mathrm{tr}\circ\hat{\rho}(\gamma)\right|\right)$ along the Fuchsian locus, and the analytic extension of this function everywhere-else.
\end{dfn}

\begin{note}
The fact that our previous geometric description and the above algebraic definition agree may be shown using the holomorphic identity theorem (see, for example, Proposition~6.5 of\cite{loustau}): both the geometrically defined length functions and its algebraic counterpart yield holomorphic functions on $\mathcal{QF}(N)$ and agree on the Fuchsian locus -- a maximal dimensional totally real analytic submanifold, and therefore must be the same function.
\end{note}

\subsubsection{For pseudo-Anosov mapping Klein bottle}

The definition of complex lengths given in Definition~\ref{dfn:qfcomplexlength} also extends to holonomy representations of pseudo-Anosov mapping Klein bottle, and are utilized in Theorem~\ref{thm:zero} and Theorem~\ref{thm:kleinbottle}. To begin with, we may restrict the holonomy representation $\phi:\pi_1(K_{(d\varphi,\iota)})\to\mathrm{PSL}(2,\mathbb{C})$ for a pseudo-Anosov mapping Klein bottle $K_{(d\varphi,\iota)}$ to the fundamental group $\pi_1(N)$ of the (non-canonical) singular surface fiber homeomorphic to $N$. We can further restrict to $\pi_1(dN)\leq \pi_1(N)$, to define a representation $d\phi=\phi|_{\pi_1(dN)}$ which is a limit of quasifuchsian representations of $\pi_1(dN)$. Lemma~3.8 (along with Claim~3.9 and Definition~3.10) of\cite{amsgroups} suffice to ensure that complex lengths are well-defined for simple curves on $dN$. This in turn means that the complex lengths of $2$-sided simple curves on $N$ are well-defined, because they lift to two distinct simple closed curves on $dN$ with the same complex length --- the fact that these two complex lengths are the same owes to them being related by the action of an orientation-preserving involution on $\mathbb{H}^3/d\phi(\pi_1(dN))$. For an arbitrary $1$-sided simple curve $\alpha\in\pi_1(N)$, its double $\alpha^2\in\pi_1(dN)$ lifts to a simple closed curve on $dN$ and hence has a well-defined complex length. We define $\ell_\alpha(\rho)$ as $\frac{1}{2}\ell_{d\rho}(\alpha^2)$ to produce a holomorphic length function $\ell_\alpha:\mathcal{QF}(N)\to\mathbb{C}$ on quasifuchsian space which evaluates to standard hyperbolic length on the Fuchsian locus.

\subsubsection{For transflected-quasifuchsian representations}
\label{sec:tqflength}
Theorem~\ref{thm:mcshanetransf} regards transflected-quasifuchsian representations. However, the notion of complex length we require in this context is less well-defined. Clearly, translation lengths \eqref{eq:translationlength} are still well-defined and constitute the real component of the complex length $\ell_\gamma$ of a curve $\gamma$ with respects to a transflected-quasifuchsian representation $\rho:\pi_1(N)\to O^+(1,3)$ as for the quasifuchsian case. The imaginary part is more problematic, however, and we first consider the case when $\alpha\in\pi_1(N)-\pi_1(dN)$. In this case, the isometry $\rho(\alpha)$ is going to be orientation-reversing and non-periheral. Therefore, its double $\alpha^2\in\pi_1(dN)$ must correspond to a non-peripheral orientation-preserving isometry $\rho(\alpha^2)$, which is to say that $\rho(\alpha^2)$ is loxodromic (including hyperbolic). This in turn implies that $\rho(\alpha)$ must be a transflection as it cannot have an fixed points in $\mathbb{H}^3$ and hence cannot be a reflection, an improper reflection or a point inversion. However, this then means that $\rho(\alpha^2)\in \mathrm{PSL}(2,\mathbb{C})=SO^+(1,3)$ is a hyperbolic transformation and any representation theoretically compatible notion of complex length for $\rho(\alpha^2)$ must have imaginary component in $\pi i\mathbb{Z}$. Since this is a discrete set and $\mathcal{TGF}(N)$ is simply connected, there is only one choice which result in a continuous function $\ell_\alpha:\mathcal{TQF}(N)\to\mathbb{C}$ whilst evaluating to the usual length function on the Fuchsian locus, namely $\ell_\alpha$ is always real.\medskip

Now consider $2$-sided curves $\gamma\in\pi_1(dN)\leq\pi_1(N)$. We know that each curve $\gamma$ lifts to two curves $\gamma_1,\gamma_2$ on $dN$. The complex lengths for $\gamma_1$ and $\gamma_2$ are both well-defined but are different. In particular, since $\gamma_1$ and $\gamma_2$ are related an orientation-reversing involution on $\mathbb{H}^3/d\rho(\pi_1(dN))$, their complex length are related by complex conjugation. Fortunately, the summands for Theorem~\ref{thm:mcshanetransf} depend only on $\mathrm{Re}\ell_{\gamma_1}=\mathrm{Re}\ell_{\gamma_2}$ and $|\mathrm{Im}\ell_{\gamma_1}|=|\mathrm{Im}\ell_{\gamma_2}|$. Specifically, this can be seen from the fact that:
\begin{align*}
\mathrm{Re}
\left(e^z+(-1)^k\right)^{-1}
=
\mathrm{Re}
\left(e^{\bar{z}}+(-1)^k\right)^{-1}
&=\tfrac{1}{2}\left(
\left(e^z+(-1)^k\right)^{-1}
+\left(e^{\bar{z}}+(-1)^k\right)^{-1}
\right)\\
&=
\frac{e^x\cos|y|+(-1)^k}
{e^{2x}+2(-1)^k e^x\cos|y|+1}\text{, where }z=x+yi.
\end{align*}

\subsubsection{For pseudo-Anosov mapping tori}

Finally, we need to contend with a notion of complex length for the statement of Theorem~\ref{thm:zeropatori}. Again, the expression of the summands means that we only need the length to be well-defined up to complex conjugation. The strategy here is essentially what we have seen so far: 
\begin{itemize}
\item
restrict the holonomy representation $\phi:\pi_1(M_\varphi)\to\mathrm{PSL}(2,\mathbb{C})$ for a pseudo-Anosov mapping torus $M_\varphi$ to the fundamental group $\pi_1(N)$ of the circle-bundle fiber $N$;
\item
further restrict to $\pi_1(dN)\leq \pi_1(N)$, to define a representation $d\phi=\phi|_{\pi_1(dN)}$ which is a limit of quasifuchsian representations of $\pi_1(dN)$;
\item
invoke Lemma~3.8 (along with Claim~3.9 and Definition~3.10) of\cite{amsgroups} to ensure that complex lengths are well-defined for simple curves on $dN$;
\item
use the same arguments as in \S\ref{sec:tqflength} to show that the lengths of $1$-sided simple closed geodesics are well-defined real functions equal to its translation length and that the lengths of $2$-sided simple closed geodesics $\gamma$ are well-defined up to complex conjugation and equal the complex lengths of either lift of $\gamma$ in $dN$.
\end{itemize}

\section{Simple geodesics on $N$}

Let $m_p\in\pi_1(N)$ denote a peripheral homotopy class going around cusp $p$ once. Normalize every $\rho\in\mathcal{QF}(N)$ so that $\rho(m_p)=\pm\left[\begin{smallmatrix}1&1\\ 0 &1\end{smallmatrix}\right]$. Consider the holonomy representation $\rho_0$ for $N$. The restriction of its limit curve $C_{\rho_0}$ to $\hat{\mathbb{C}}-\{\infty\}$ is precisely the real axis $\mathbb{R}\subset\hat{\mathbb{C}}$, and $\mathbb{R}/\rho_0(m_p)=\mathbb{R}/\mathbb{Z}$ canonically identifies with the set $\mathbb{S}^1_p$ of complete (oriented) geodesics in $N$ emanating from $p$. On the other hand, Proposition~\ref{thm:double} tells us that for an arbitrary quasifuchsian representation $\rho$, there are quasiconformal maps which $\pi_1(N)$-equivariantly identify $C_{\rho_0}$ with $C_\rho$ and hence identify $(C_\rho-\{\infty\})/\rho(m_p)$ with $\mathbb{S}^1_p$. We pay particular attention to two subsets of $\mathbb{S}^1_p$:
\begin{itemize}
\item
$\vec{\triangle}$: the set of oriented bi-infinite simple geodesics on $N$ with both source and sink based at $p$;
\item
$\vec{\mathcal{G}}$: the set of oriented simple complete geodesics on $N$ with source based at $p$.
\end{itemize}
The set $\vec{\triangle}$ is contained in $\vec{\mathcal{G}}$, and we may topologize both of these spaces via the subspace topology on $\mathbb{S}^1_p$.

\begin{dfn}[Ideal geodesics]
\label{defn:sidedness}
Let $\triangle$ denote the set of unoriented ideal geodesics on $N$ with both ends at cusp $p$. We say that that an ideal geodesic $\sigma$ in $\triangle$ is a \emph{$1$-sided} (or \emph{$2$-sided}) ideal geodesic if, upon filling in the cusp $p$ on the surface $N$, $\sigma$ completes to a 1-sided (resp. 2-sided) curve. We denote the collection of $1$-sided ideal geodesics by $\triangle_1$ and the collection of $2$-sided geodesics by $\triangle_2$.
\end{dfn}

\subsection{Fattening simple geodesics}

Any (simple) ideal geodesic $\sigma\in\triangle$ may be fattened up into an (open) geodesically bordered surface as follows: any sufficiently small $\epsilon$-neighborhood of a $2$-sided $\sigma$ is a pair of pants (Figure~\ref{fig:fattening} -- left), whereas any small $\epsilon$-neighborhood of a $1$-sided $\sigma$ is topological equivalent to a $1$-holed M\"obius band (Figure~\ref{fig:fattening} -- right). Isotoping the boundaries of these $\epsilon$-fattened surfaces until they are geodesically bordered results in elements of $\mathcal{S}_i(N)$ for $i$-sided ideal geodesics $\sigma\in\triangle_i$. This fattening procedure is well-defined as any two sufficiently small $\epsilon$-neighborhoods are related by a deformation retract, and this gives us an injective function $\mathrm{Fat}:\triangle\rightarrow\mathcal{S}(N)$.

\begin{prop}
\label{thm:bijection}
The $\mathrm{Fat}$ map gives a topologically defined bijection between $\mathcal{S}(N)$ and the collection $\triangle$ of (simple) ideal arcs on $N$ with both ends up $p$, such that each of the two curves $\{\alpha,\beta\}\in\mathcal{S}(N)$ is freely homotopic to its corresponding ideal geodesic $\sigma\in\triangle$ when $\alpha,\beta$ and $\sigma$ are regarded as simple closed curves on $N\cup\{p\}$, that is: the surface $N$ with cusp $p$ filled in.
\end{prop}

\begin{figure}[h!]
\centering
\includegraphics[width=0.75\textwidth]{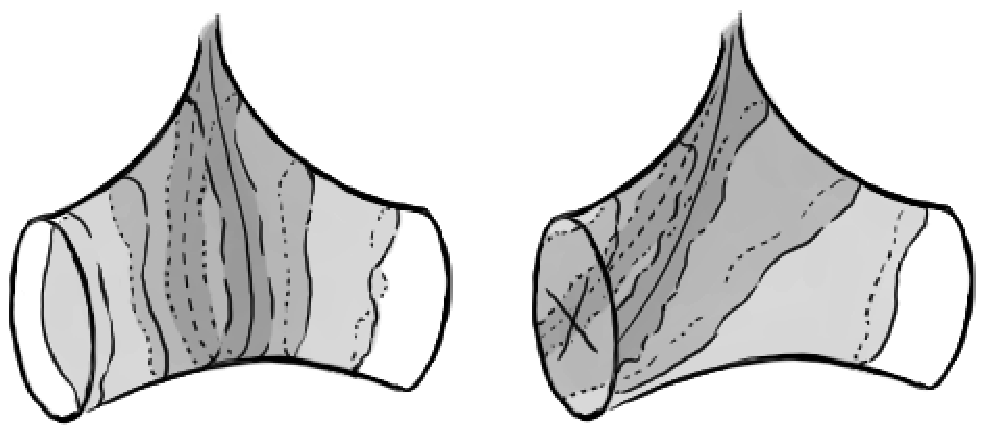}
\caption{(left) fattening a $2$-sided ideal geodesic to a pair of pants; (right) fattening a $1$-sided ideal geodesic to a $1$-holed M\"obius band.}\label{fig:fattening}
\end{figure}

\begin{proof}
The descriptions of $\alpha$ and $\beta$ are simple topological consequences of the fattening procedure and we only prove the statement that $\mathrm{Fat}$ is a bijection. The fact that $\mathrm{Fat}$ is a surjection is clear from Figure~\ref{fig:fattening} and the existence of geodesic representatives for homotopy classes of ideal arcs  on hyperbolic surfaces. The fact that $\mathrm{Fat}$ is an injection on $\triangle_2$ follows from the fact that for every embedded pair of pants $P\in\mathcal{S}_2(N)$ there is a unique simple ideal geodesic, with both cusps up $p$, which lies completely on $P$. For injectivity on $\triangle_1$, we remark that for any embedded $1$-holed M\"obius band $M$ with cusp $p$, there are precisely three (unoriented) simple ideal geodesics on $M$ with both ends going up $p$ (Lemma~\ref{thm:fourteen}). Two of these are 2-sided and hence correspond to elements of $\triangle_2$ and only one is 1-sided. 
\end{proof}

\begin{note}
Thanks to the above result, we may regard elements of $\vec{\triangle}$ triples as $\{\alpha,\beta;\epsilon\}$, where $\{\alpha,\beta\}$ is an element of $\mathcal{S}(N)$ and $\epsilon\in\{-,+\}=\{\pm\}$ (arbitrarily) specifies the orientation of the bi-infinite ideal geodesic.
\end{note}

\begin{note}\label{note:equivariance}
The fattening procedure is a fundamentally topological construction, and hence every homeomorphism $\varphi:N\rightarrow N$ acts equivariantly on $\triangle$ and $\mathcal{S}(N)$ with respect to the fattening map $\mathrm{Fat}:\triangle\rightarrow\mathcal{S}(N)$.
\end{note}

\begin{lem}\label{thm:fourteen}
There are precisely fourteen elements of $\vec{\mathcal{G}}\subset\mathbb{S}^1_p$ on any (open) $1$-holed M\"obius band $M$ containing cusp $p$ and one other geodesic border. Moreover,
\begin{itemize}
\item
these fourteen oriented geodesics are naturally grouped as seven pairs of geodesics, where each pair is related by the reflection involution on $M$.
\item
three of these pairs are of elements of $\vec{\triangle}$ and each pair consists of the same ideal geodesic with its two opposing orientations. The inner pair $\{\lambda^-,\lambda^+\}$ of geodesics are oriented versions of the $1$-sided geodesic $\lambda$ shown in the top left diagram in Figure~\ref{fig:interlace}. The outer two pairs $\{\lambda^-_\alpha,\lambda^+_\alpha\}$ and $\{\lambda^-_\beta,\lambda^+_\beta\}$ are oriented versions of the two $2$-sided ideal geodesics $\lambda_\alpha$ and $\lambda_\beta$ depicted in the bottom left diagram in Figure~\ref{fig:interlace}.
\item
the remaining four pairs are of elements of $\vec{\mathcal{G}}-\vec{\triangle}$ (blue geodesics in Figure~\ref{fig:interlace}) consisting of simple bi-infinite geodesics with one end spiraling to some simple closed geodesic. Specifically, two of the pairs $\{\mu^-_\alpha,\mu^+_\alpha\}$ and $\{\mu^-_\beta,\mu^+_\beta\}$ spiral to the two interior (1-sided) simple geodesics $\alpha,\beta$ on $M$ and two of the pairs $\{\nu^-_\alpha,\nu^+_\alpha\}$ and $\{\nu^-_\beta,\nu^+_\beta\}$ spiral to the (2-sided) non-cusp boundary of $M$.
\item
the four pairs of geodesics in $\vec{\mathcal{G}}-\vec{\triangle}$ and the three pairs of geodesics in $\vec{\triangle}$ interlace each other as elements of $\mathbb{S}^1_p$ as per Figure~\ref{fig:intervals}.
\item
as per Figure~\ref{fig:intervals}, each of the six elements of $\{\lambda^\pm,\lambda^\pm_\alpha,\lambda^\pm_\beta\}\subset\vec{\triangle}$ is adjacent to two open intervals which constitute connected components of $\mathbb{S}^1_p-\vec{\mathcal{G}}$. All twelve such open intervals are distinct.
\end{itemize}
\end{lem}

\begin{figure}[h!]
\centering
\includegraphics[width=0.8\textwidth]{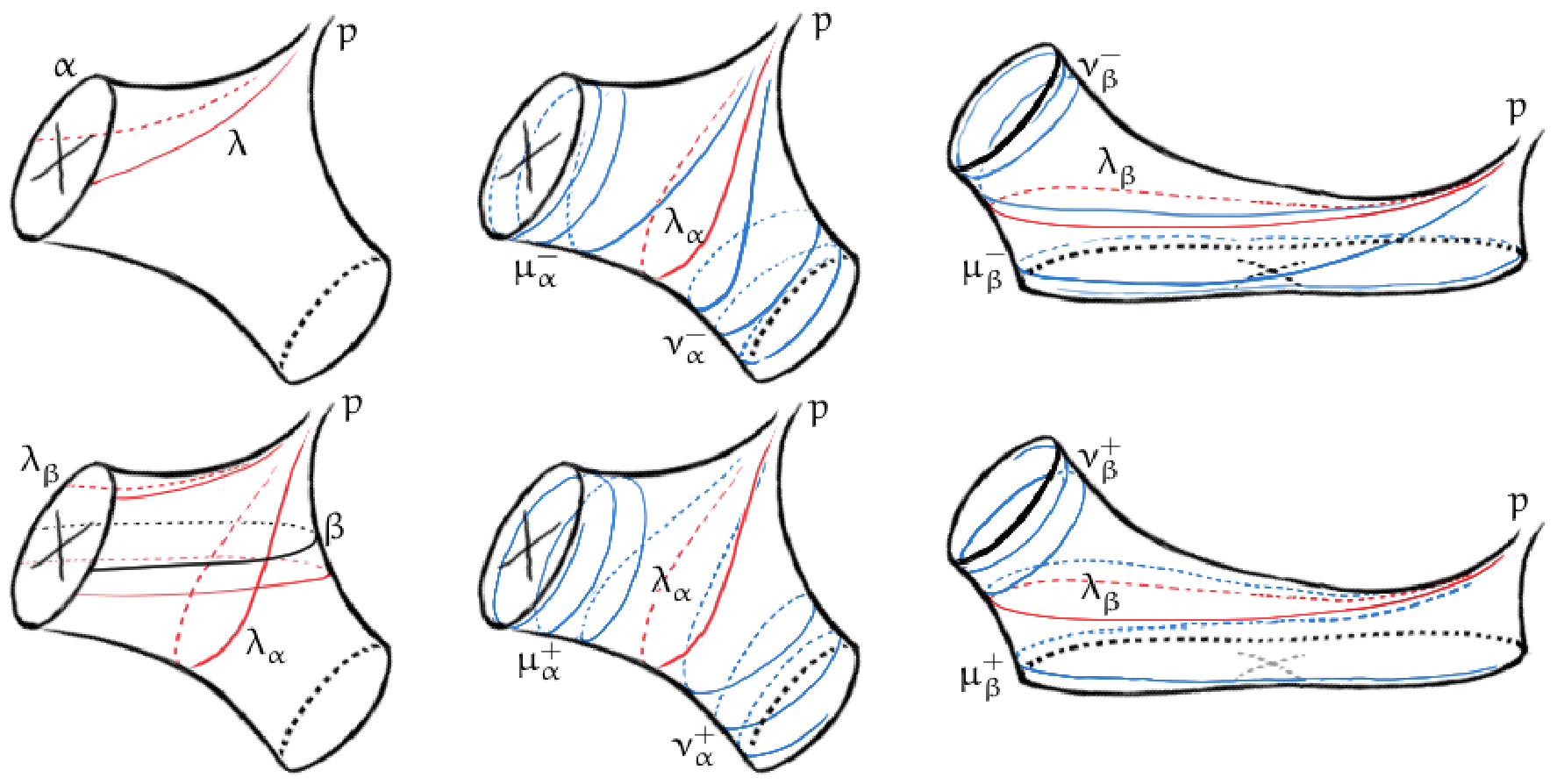}
\caption{(left column) all three unoriented (hence six oriented) simple ideal geodesics on $M$ with both ends up the cusp $p$; (middle column) all the simple ideal geodesics which do not intersect $\alpha$; (right column) all the simple ideal geodesics which do not intersect $\beta$.}\label{fig:interlace}
\end{figure}

\begin{figure}[h!]
\centering
\includegraphics[width=0.4\textwidth]{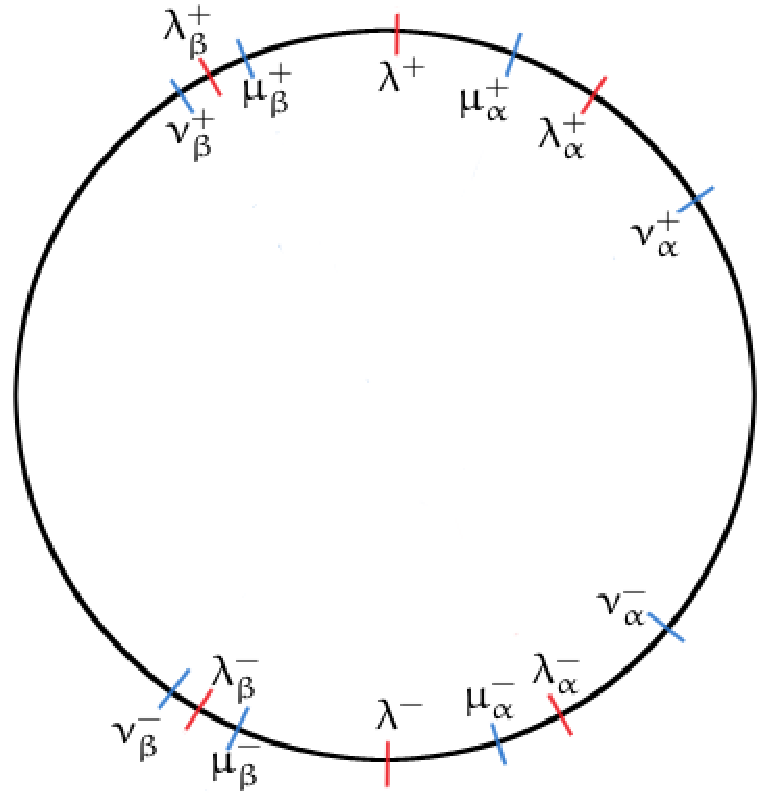}
\caption{A depiction of all fourteen oriented geodesics in $\vec{\mathcal{G}}$ which lie on $M$, as a subset of the set $\mathbb{S}^1_p$ of all directions emanating from cusp $p$.}
\label{fig:intervals}
\end{figure}

\begin{proof}
The existence of these seven pairs of oriented geodesics is due to the existence of unique geodesics representatives of curves on hyperbolic surfaces. Provided that we believe that these are all the simple ideal geodesics on $M$, all five placement properties stated in the lemma are easily deduced from
\begin{itemize}
\item
 Figure~\ref{fig:interlace}, 
 \item
 the uniqueness of geodesic representatives for homotopy classes of ideal paths, and 
 \item
 the fact that these geodesics do not intersect if they have homotopy equivalent representative paths which do not intersect. 
 \end{itemize}
 Therefore, the only thing that we need to prove is that there are no other ideal geodesics on $M$. First note that since $\{\lambda^\pm_\alpha,\lambda^\pm_\beta\}$ lie on pairs of pants contained in $M$, the eight open intervals adjacent to these four oriented 2-sided geodesics all correspond to self-intersecting geodesics (see Theorem~9 of\cite{mcshane_allcusps}).  The remaining four intervals correspond to geodesics which are launched in between $\lambda^-$ and $\mu^-_\alpha$ or $\lambda^-$ and $\mu^-_\beta$ or $\lambda^+$ and $\mu^+_\alpha$ or $\lambda^+$ and $\mu^+_\beta$. These four intervals have equivalent roles to one another, and so we only consider one of them.\medskip

\begin{figure}[h!]
\centering
\includegraphics[width=0.25\textwidth]{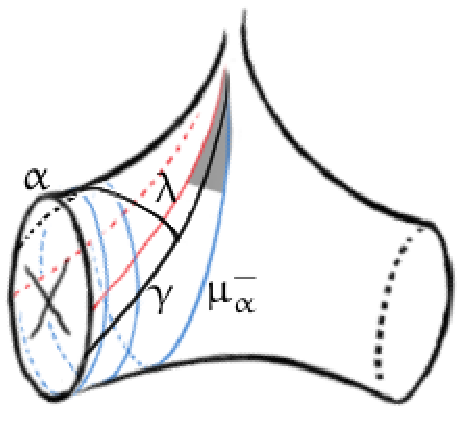}
\caption{The shaded gray region represents one of the four intervals of directions considered in the previous paragraph.}\label{fig:selfintersect}
\end{figure}

\begin{figure}[h!]
\centering
\includegraphics[width=0.5\textwidth]{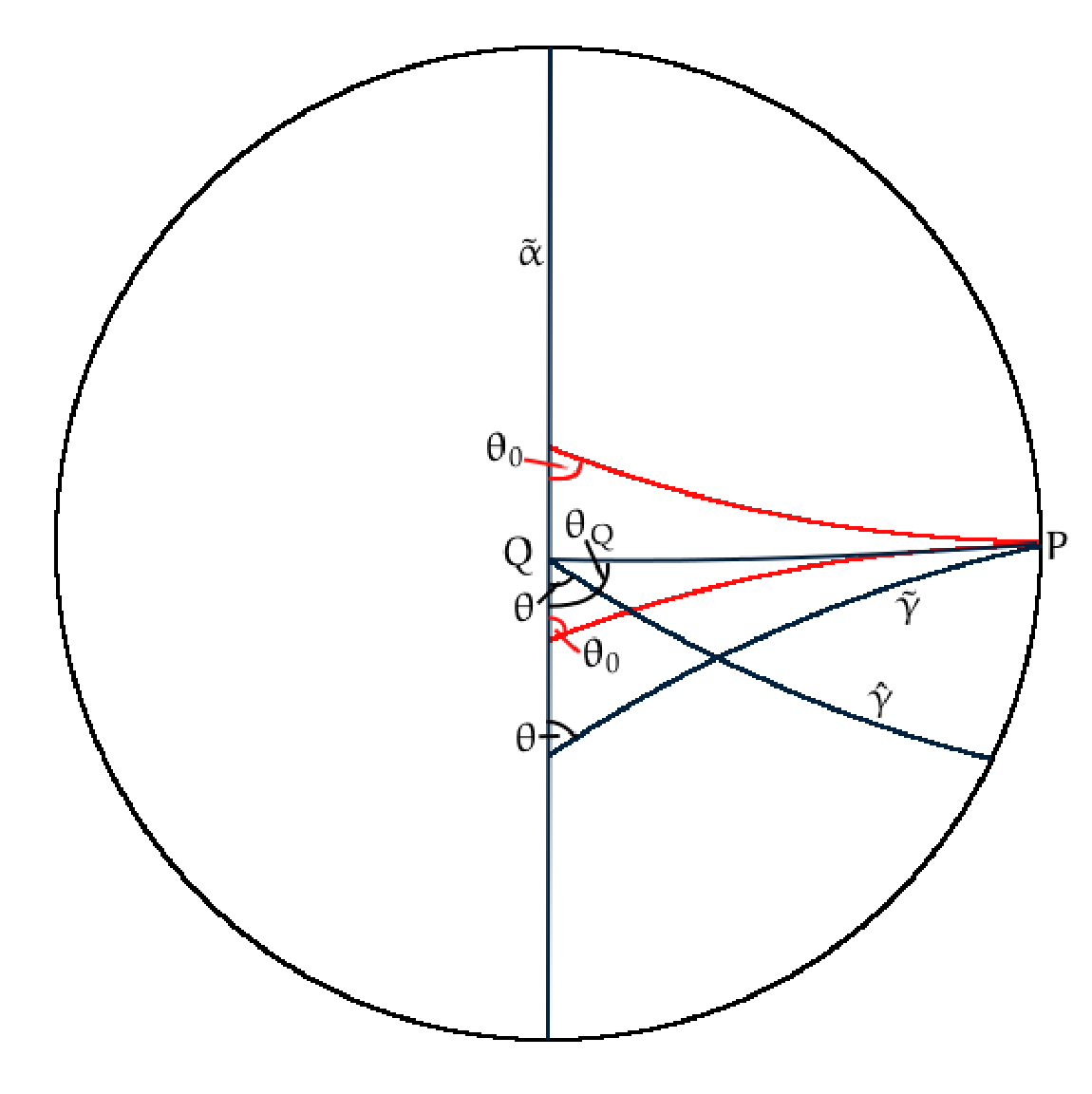}
\caption{The geodesic $\gamma$ must self-intersect.}\label{fig:selfintersecttrig}
\end{figure}

Any geodesic $\gamma$ launched within one of these four intervals (gray region in Figure~\ref{fig:selfintersect}) will necessarily hit $\alpha$ (without loss of generality) at an angle $\theta<\theta_0$, where $\theta_0$ is the angle between $\alpha$ and $\lambda$. We see in Figure~\ref{fig:selfintersecttrig} that a lift $\tilde{\gamma}$ of $\gamma$ is launched from a lift $P\in\partial_\infty\mathbb{H}^2$ of cusp $p$, hits a lift $\tilde{\alpha}$ of $\alpha$ and re-emerges on $\tilde{\alpha}$ as $\hat{\gamma}$ shifted by a distance of $\ell_\alpha$ along $\tilde{\alpha}$. Denote the point of re-emergence by $Q$. A little hyperbolic trigonometry (one may use, for example, Theorem~2.2.2 of\cite{buser}) suffices to show that the angle $\theta_Q$ between the geodesic $\overline{PQ}$ and $\tilde{\alpha}$ is strictly greater than $\theta_0$. Since $\hat{\gamma}$ re-emerges from $Q$ at an angle $\theta<\theta_0<\theta_Q$ within the triangle bordered by $\tilde{\alpha},\tilde{\gamma}$ and $\overline{PQ}$ it must eventually hit one of the sides of this triangle. It cannot hit $\overline{PQ}$ or $\tilde{\alpha}$ as that would form hyperbolic $2$-gons, and therefore must intersect $\tilde{\gamma}$. This intersection descends to a self-intersection point on $\gamma$.
\end{proof}

\subsection{The classification of simple geodesics}

\begin{thm}[Classification of simple geodesics]
\label{thm:simplecharacterization}
The following three types of behaviors partition $\vec{\mathcal{G}}$:
\begin{enumerate}
\item
$\gamma$ is an isolated point in $\vec{\mathcal{G}}$ iff. either $\gamma$ has both ends up cusps or if it spirals to a $1$-sided geodesic;
\item
$\gamma$ is a boundary point of $\vec{\mathcal{G}}$ iff. $\gamma$ spirals towards a $2$-sided simple closed geodesic;
\item
$\gamma$ is neither a boundary nor an isolated point of $\vec{\mathcal{G}}$ iff. $\gamma$ spirals toward a (minimal) geodesic lamination which is not a simple closed geodesic.
\end{enumerate}
\end{thm}

\begin{proof}
The proof of this result is fairly similar to its orientable-case counterpart and we only outline most of the necessary steps. To begin with, we know that the $\omega$-limit set of the constant speed flow along an oriented geodesic ray $\gamma$ is either a minimal geodesic lamination or empty (i.e.: $\gamma$ goes up a cusp).\medskip

Observe that when the $\omega$-limit of $\gamma\in\vec{\mathcal{G}}$ is a $2$-sided geodesic $\alpha$, the geodesic $\gamma$ fattens up to a geodesically bordered pair of pants homotopy equivalent to any sufficiently small $\epsilon$-neighborhood of $\gamma$. Thus, by Lemma~\ref{thm:fourteen}, there is at least one open interval in $\mathbb{S}^1_p-\vec{\mathcal{G}}$ adjacent to $\gamma$ and hence $\gamma$ is either an isolated point or a boundary point in $\vec{\mathcal{G}}$. Let $\sigma$ be a simple ideal geodesic which intersects $\alpha$, then the sequence of ideal geodesics obtained by Dehn-twisting $\sigma$ along $\alpha$ is a sequence in $\vec{\mathcal{G}}$ approaching $\gamma$. Therefore, any geodesic $\gamma$ which spirals to a $2$-sided geodesic is a boundary point of $\vec{\mathcal{G}}$.\medskip

Next we consider the case when the $\omega$-limit of $\gamma$ is a geodesic lamination $\omega(\gamma)$ which is not a simple closed geodesic. We follow Mirzakhani's proof (Theorem~4.6 of\cite{mirz_simp}) and show that $\gamma$ is not an isolated point by approximating it by a sequence of geodesics $\{\gamma_i\}$ in $\vec{\mathcal{G}}$ which each spiral to a distinct simple closed geodesic. Construct a sequence of quasigeodesics $\hat{\gamma}_i\in\vec{\mathcal{G}}$ as follows: fix a sequence of positive numbers $\{\epsilon_i\}$ converging to $0$. For each $\epsilon_i$, traverse along $\gamma$ until you come to a point $\gamma(t_1)$ along $\gamma$ within distance $\epsilon_i$ of a previous point $\gamma(t_0)$ on $\gamma$ so that
\begin{itemize}
\item
the geodesic arc $\eta$ between $\gamma(t_0)$ and $\gamma(t_1)$ does not intersect $\gamma|_{(-\infty,t_1]}$ (except at its ends),
\item
the arc $\eta$ is within $\epsilon_i$ radians of being orthogonal to $\gamma$ at its two ends, and
\item
the unit tangent vectors $\gamma'(t_0)$ and $\gamma'(t_1)$ are almost parallel; i.e.: the parallel transport of $\gamma'(t_1)$ to $\gamma(t_0)$ is within $\epsilon_i$ radians of $\gamma'(t_0)$.
\end{itemize}
Take the quasigeodesic $\hat{\gamma}_i$ to be the path which traverses along $\gamma$ until time $t_0$ and then indefinitely traverses the broken geodesic loop formed by joining $\eta$ and $\gamma|_{[t_0,t_1]}$ and let $\gamma_i\in\vec{\mathcal{G}}$ be the simple geodesic representative of $\hat{\gamma}_i$. The sequence $\{\gamma_i\}$ approaches $\gamma$. Moreover, depending on whether $\eta$ is chosen to to turn clockwise or anticlockwise when one goes from $\gamma(t_1)$ to $\gamma(t_0)$, we may construct $\{\gamma_i\}$ to approach $\gamma$ from both sides. Therefore $\gamma$ cannot be a boundary point either.\medskip

The previous two paragraphs tell us that the only possible isolated points in $\vec{\mathcal{G}}$ are ideal geodesics $\gamma$ with both ends up cusps or geodesics $\gamma$ which spiral toward $1$-sided simple closed geodesics. Conversely, any such $\gamma$ is an isolated point. If $\gamma$ is an ideal geodesic with both ends up the same cusp (we may assume cusp $p$ wlog), then it is isolated by Lemma~\ref{thm:fourteen}. If $\gamma$ goes between different cusps, then it fattens to an embedded pair of pants and by McShane's original proof (Theorem~9 of\cite{mcshane_allcusps}), it must be an isolated point. If $\gamma$ spirals to a $1$-sided simple closed geodesic $\alpha$, then $\gamma$ fattens to an embedded cusped M\"obius band, and is isolated by Lemma~\ref{thm:fourteen}. This proves statement~$1$. Since geodesics $\gamma$ which spiral to a geodesic lamination which is not a closed geodesic cannot be boundary points, this proves statement~$2$ and hence statement~$3$.
\end{proof}

\begin{cor}
The set $\vec{\mathcal{G}}-\vec{\triangle}$ is a Cantor set of measure $0$.\label{thm:cantor}
\end{cor}

\begin{proof}
This follows as a consequence of Theorem~\ref{thm:simplecharacterization} because $\vec{\mathcal{G}}-\vec{\triangle}$ is a (non-empty) perfect, compact, totally disconnected metric space. The fact that it has measure $0$ is a consequence of the Birman-Series geodesic sparsity theorem\cite{birmanseries}.
\end{proof}

\begin{note}
The set $\vec{\mathcal{G}}-\vec{\triangle}$ may be obtained by iteratively process of removing open intervals surrounding $\vec{\triangle}$. In particular, no remnant (i.e.: unremoved) closed interval at any given finite step in this process remains unperturbed --- it will, at some stage, have some open interval removed from its ``center''. In fact, it is possible to order the removal of these open sets in much the same way as one might when constructing the usual Cantor set, and in this regard, the fact that $\vec{\mathcal{G}}-\vec{\triangle}$ is a Cantor set is very natural.
\end{note}

\begin{note}
\label{note:summeaning}
Theorem~\ref{thm:simplecharacterization} tells us that every isolated point in $\vec{\mathcal{G}}$ is surrounded by two intervals (one on the left, one on the right) of ``directions" in $\mathbb{S}^1_p$ where geodesics shot out in those directions must self-intersect. In fact, every summand in the Fuchsian McShane identity may be interpreted as the measure of some such interval-pair.
\end{note}

\section{Identities for quasifuchsian representations}

We begin by proving the McShane identity for quasifuchsian representations of nonorientable surface groups by first showing that the series constituting one of the sides of our McShane identity yields a holomorphic function, and then invoking a version of the identity theorem for holomorphic functions on complex manifolds to assert that the identity holds over the entire quasifuchsian character variety. 

\subsection{McShane identity for quasifuchsian representations}

\begin{prop}\label{thm:holomorphicity}
The series 
\begin{align}
H(\rho):=
&\sum_{\{\alpha,\beta\}\in\mathcal{S}(N)}\left(e^{\frac{1}{2}(\ell_{\alpha}(\rho)+\ell_{\beta}(\rho))}+(-1)^{\alpha\cdot\beta}\right)^{-1}\label{eq:holomorphicseries}
\end{align}
defines a well-defined holomorphic function on $\mathcal{QF}(N)$. 
\end{prop}

\begin{proof}
We use the fact that a pointwise convergent sequence of holomorphic functions that is uniformly convergent on all compact sets converges to a holomorphic function. To begin with, we specify an ordering on the summands for \eqref{eq:holomorphicseries} and consider the sequence of partial sums for this series.\\
\newline
Let $R\subset\mathbb{H}=\tilde{N}$ be a fundamental domain for $N$ such that $R$ is a finite sided geodesic ideal polygon. The boundary $\partial R$ of $R$ projects to a collection of disjoint ideal geodesics $\pi(\partial R)$ on $N$, and every essential simple closed geodesic pair $\{\gamma_1,\gamma_2\}\in\mathcal{S}=\mathcal{S}_1\cup\mathcal{S}_2$ intersects $\pi(\partial R)$ transversely and nontrivially. Thus, to any collection of geodesics $\{\gamma_1,\ldots,\gamma_k\}$, we may assign a positive integer $\norm{\{\gamma_1,\ldots,\gamma_k\}}$ denoting the total number of geodesic segments that $\{\gamma_1,\ldots,\gamma_k\}$ splits into when cut along $\pi(\partial R)$. Order the elements of $\mathcal{S}$ as a sequence $(\{\gamma_1,\gamma_2\}_i)_{i\in\mathbb{N}}$ with nondecreasing $\norm{\{\gamma_1,\gamma_2\}_i}$ and consider the function $Q_0:\mathbb{N}\rightarrow\mathbb{N}$ counting the number of $\{\gamma_1,\gamma_2\}$ with $\norm{\{\gamma_1,\gamma_2\}}\leq n$:
\begin{align*}
Q_0(n):=\mathrm{Card}\left\{\{\gamma_1,\gamma_2\}\in\mathcal{S}\mid \norm{\{\gamma_1,\gamma_2\}}\leq n\right\}.
\end{align*}
It is clear that $Q_0(n)$ is bounded above by $P_0(n)^2$, for
\begin{align*}
P_0(n):=\mathrm{Card}\left\{\gamma\text{ - (the image of) a simple closed geodesic on }dN \mid \norm{\{\gamma\}}\leq n\right\}.
\end{align*}
The function $P_0$ is in turn bounded above by polynomial (Lemma~2.2 of\cite{birmanseries}), and therefore $Q_0(n)$ is bounded above by a polynomial in $n$.\\
\newline
Consider the following sequence of partial sums:
\begin{align}
H_n(\rho):=\sum_{\substack{\{\gamma_1,\gamma_2\}_i\\\text{for }i\leq n}}
\left(e^{\frac{1}{2}(\ell_{\gamma_1}(\rho)+\ell_{\gamma_2}(\rho))}+(-1)^{\gamma_1\cdot\gamma_2}\right)^{-1}.
\end{align}
Since the length functions $\ell_\gamma$ are holomorphic on $\mathcal{T}(dN)=\mathcal{QF}(N)$, each partial sum $H_n$ is a holomorphic function on $\mathcal{QF}(N)$.\newline
\\
It remains to show that for any compact set $C\subset\mathcal{QF}(N)$, the sequence of functions $(H_i)_{i\in\mathbb{N}}$ is uniformly absolutely convergent. We utilize the following fact (Lemma~5.2 of\cite{amsgroups}): let $\rho_0$ be a Fuchsian representation for $N$, then for every compact set $C\subset\mathcal{QF}(N)$, there exist $C$-dependent constants $c>0$ and $k>0$ such that for all $\gamma\in\mathcal{S}$ and $\rho\in C$,
\[\tfrac{c}{k}\norm{\gamma}\leq \tfrac{1}{k}\ell_\gamma(\rho_0)\leq \mathrm{Re}(\ell_\gamma(\rho)).\]
Therefore, we obtain the following comparisons:
\begin{align}
\sum_{\substack{\{\gamma_1,\gamma_2\}_i\\\text{for all }i}}
\abs{\left(e^{\frac{1}{2}(\ell_{\gamma_1}(\rho)+\ell_{\gamma_2}(\rho))}+(-1)^{\gamma_1\cdot\gamma_2}\right)^{-1}}
\leq
\sum_{\substack{\{\gamma_1,\gamma_2\}_i\\\text{for all }i}}
\left(e^{\frac{c}{2k}\norm{\{\gamma_1,\gamma_2\}}}-1\right)^{-1}
\leq
\sum_{m=1}^\infty
\frac{Q_0(m)-Q_0(m-1)}{e^{\frac{c}{2k}m}-1}.\label{eq:uniformbound}
\end{align}
The fact that \eqref{eq:uniformbound} converges ensures that $H(\rho):=\lim_{n\to\infty} H_n(\rho)$ is well-defined and that the sequence $(H_i)$ is uniformly absolutely convergent. Finally, the absolute convergence of this series ensures that this limit is independent of the ordering we placed on $\mathcal{S}$ when summing the series.
\end{proof}

\begin{repeatthm2'}[Identity for quasifuchsian representations of nonorientable surface groups]
Given a nonorientable cusped hyperbolic surface $N$ and a quasifuchsian representation $\rho:\pi_1(N)\to \mathrm{PSL}(2,\mathbb{C})$, define $\mathcal{S}(N)$ to be the set of embedded pairs of pants and $1$-holed M\"{o}bius bands containing cusp $p$ (Note~\ref{note:1}). Then, 
\begin{align*}
\sum_{\{\alpha,\beta\}\in\mathcal{S}(N)}\left(e^{\frac{1}{2}(\ell_{\alpha}(\rho)+\ell_{\beta}(\rho))}+(-1)^{\alpha\cdot\beta}\right)^{-1}
=\frac{1}{2}.
\end{align*}
where $\ell_\gamma(\rho)$ is the \emph{complex length} of $\gamma$ (see \S\ref{sec:hyperbolic}).
\end{repeatthm2'}

\begin{proof}
By Proposition~\ref{thm:holomorphicity}, we know that $H(\cdot)$ defines a holomorphic function on $\mathcal{QF}(N)$. Moreover, we know that $H\equiv\frac{1}{2}$ on the Fuchsian locus of $\mathcal{QF}(N)$, which is a totally real analytic submanifold of maximal dimension. Thus, the identity theorem (see, e.g.: Proposition~6.5 of\cite{loustau}) tells us that $H(\rho)=\frac{1}{2}$ for every $\rho\in\mathcal{QF}(N)$, giving us the desired identity.
\end{proof}

\subsection{Identity for horo-core annuli}

Given a quasifuchsian representation $\rho:\pi_1(N)\rightarrow\mathrm{PSL}(2,\mathbb{C})$, consider the convex core of its corresponding quasifuchsian 3-manifold $\mathbb{H}/\rho(\pi_1(N))$. Any sufficiently small horospherical cross-section of the cusp $p$ in $\mathbb{H}/\rho(\pi_1(N))$ is a flat annulus. 

\begin{dfn}[horo-core annulus]
The conformal structure of this annulus is independent of the chosen horosphere (given that it is sufficiently small). We refer to this flat annulus, up to homothety, as the \emph{horo-core annulus} of $\rho$ at $p$.
\end{dfn}

Let $m_p\in\pi_1(N)$ denote a peripheral homotopy class going around cusp $p$ once. Normalize every $\rho$ so that $\rho(m_p)=\pm\left[\begin{smallmatrix}1&1\\ 0 &1\end{smallmatrix}\right]$, since the limit curve $C_\rho$ is invariant under translation by $1$ (i.e.: the action of $\rho(m_p)$), there must be points on the limit curve $C_\rho$ realizing the minimum and the maximum height (i.e.: imaginary component) of $C_\rho$ on $\mathbb{C}$. 

\begin{dfn}[Width partition of $\vec{\triangle}$]
\label{dfn:width}
Let $\tilde{z}_-$ and $\tilde{z}_+$ respectively be a lowest point and a highest point on $C_\rho$ and let $z_-,z_+$ denote their projected images on $\mathbb{S}^1_p=C_\rho/\mathbb{Z}$. The points $z_\pm$ define a bipartition of $\vec{\triangle}$ as follows, let:
\begin{itemize}
\item 
$\vec{\triangle}^+(\rho)$ denote the subset of $\vec{\triangle}$ composed of simple bi-infinite geodesics with launching directions in the half-open interval $[z_-,z_+)$ (oriented with respect to $m_p$);
\item
$\vec{\triangle}^-(\rho)$ denote the subset of $\vec{\triangle}$ composed of simple bi-infinite geodesics with launching directions in the half-open interval $[z_-,z_+)$ (also oriented with respect to $m_p$).
\end{itemize}
We call any bipartition $(\vec{\triangle}^+(\rho),\vec{\triangle}^-(\rho))$ obtained from such a process a \emph{width partition}.
\end{dfn}

The main result of this section is the following identity for the modulus of the horo-core annulus of a quasifuchsian representation:

\begin{thm}[Horo-core annulus identity]
\label{thm:coremodulus}
Given a width partition $(\vec{\triangle}^+(\rho),\vec{\triangle}^-(\rho))$, the modulus $\mathrm{mod}_p(\rho)$ of the horo-core annulus at $p$ of a quasifuchsian representation $\rho$ is given by:
\begin{align}
\mathrm{mod}_p(\rho)
=&\mathrm{Im}
\sum_{\{\alpha,\beta;\epsilon\}\in\triangle^+(\rho)}\left(e^{\frac{1}{2}(\ell_{\alpha}(\rho)+\ell_{\beta}(\rho))}+(-1)^{\alpha\cdot\beta}\right)^{-1}\label{eq:modtop}\\
=-&\mathrm{Im}
\sum_{\{\alpha,\beta;\epsilon\}\in\triangle^-(\rho)}\left(e^{\frac{1}{2}(\ell_{\alpha}(\rho)+\ell_{\beta}(\rho))}+(-1)^{\alpha\cdot\beta}\right)^{-1}.\label{eq:modbot}
\end{align}
\end{thm}

In order to prove this, we first establish a mild nonorientable generalization of Akiyoshi-Miyachi-Sakuma's Theorem~2.3 in\cite{amsgroups}.

\subsubsection{Width formula}

Let $\xi,\eta\in\vec{\mathcal{G}}-\vec{\triangle}$ denote two oriented simple bi-infinite geodesics emanating from the cusp $p$, then the pair $\{\xi,\eta\}$ bipartitions the set $\vec{\mathcal{\triangle}}$, composed of all oriented simple bi-infinite geodesic arcs on $N$ with both ends at $p$, into 
the following subsets:
\begin{itemize}
\item
$\vec{\mathcal{\triangle}}^\xi_\eta$ consisting of all the geodesic arcs in $\vec{\mathcal{\triangle}}$ which are launched (along the orientation of $m_p$) between $\xi$ (inclusive) and $\eta$ (exclusive), and
\item
$\vec{\mathcal{\triangle}}^\eta_\xi$ consisting of all the geodesic arcs in $\vec{\mathcal{\triangle}}$ which are launched (along the orientation of $m_p$) between $\eta$ (inclusive) and $\xi$ (exclusive).
\end{itemize}

We have hitherto regarded $\xi$ and $\eta$ as oriented simple bi-infinite geodesics on $N$ emanating from the cusp $p$, and we now introduce an alternative interpretation of these symbols for the remainder of this paper. This is a mild form of notation abuse introduced for the sack of notational simplicity. 

\begin{note}[Reinterpretation of geodesic rays]
\label{note:reinterpret}
Given a quasifuchsian representation $\rho$, there is a natural identification between the limit curve $C_\rho$ of $\rho$ and the limit curve $C_{\rho_0}$ of $\rho_0$ via the quasiconformal uniformization map on the ideal boundary of $\mathbb{H}^3$ (this can also be done via the ideal boundary for relatively hyperbolic groups). Moreover, this identification is is equivariant with respects to the action of $m_p\in\pi_1(N)$, and so we have:
\[
\xi,\eta\in\vec{\mathcal{G}}-\vec{\triangle}
\subset\mathbb{S}_p^1
=\mathbb{R}/\mathbb{Z}
=(C_{\rho_0}-\{\infty\})/\rho_0(m_p)
\cong(C_\rho-\{\infty\})/\rho(m_p),
\]
thereby allowing us to regard $\xi,\eta$ as geodesic rays on $(C_\rho-\{\infty\})/\rho(m_p)$. There are $\mathbb{Z}$-lifts, $\{\xi_k\}$ and $\{\eta_k\}$, respectively of $\xi$ and $\eta$ on $C_\rho-\{\infty\}$ ordered so that $\{\xi_k\}$ and $\{\eta_k\}$ interlace each other along $C_\rho-\{\infty\}$ as
\[
\ldots \xi_{-2},\eta_{-2},\xi_{-1},\eta_{-1},\xi_0,\eta_0,\xi_1,\eta_1,\xi_2,\eta_2,\ldots
\]
The segment of $C_\rho-\{\infty\}$ going from $\xi_0$ (inclusive) to $\xi_1$ (exclusive) is precisely one lift of $(C_\rho-\{\infty\})/\rho(m_p)$ and we identify
\begin{itemize}
\item
$\vec{\mathcal{\triangle}}^\xi_\eta$ with all the lifts on $C_\rho-\{\infty\}$ of elements of $\vec{\mathcal{\triangle}}^\xi_\eta$ lying between $\xi_0$ (inclusive) and $\eta_0$ (exclusive);
\item
$\vec{\mathcal{\triangle}}^\eta_\xi$ with all the lifts on $C_\rho-\{\infty\}$ of elements of$\vec{\mathcal{\triangle}}^\eta_\xi$  lying between $\eta_0$ (inclusive) and $\xi_1$ (exclusive).
\end{itemize}
This re-interprets the elements of $\vec{\mathcal{\triangle}}^\xi_\eta$ and $\vec{\mathcal{\triangle}}^\eta_\xi$ as geodesic rays in $\mathbb{H}^3$ going from $\{\infty\}$ (which is a lift of $p$) to points in $C_\rho-\{\infty\}$.
\end{note}

\begin{lem}[Width formula]\label{thm:height}
Given a quasifuchsian representation $\rho$ normalized so that the boundary holonomy of $m_p$ is given by $\pm\left[\begin{smallmatrix} 1 & 1 \\ 0 & 1\end{smallmatrix}\right]$. The function $w^\xi_\eta:\mathcal{QF}(N)\rightarrow\mathbb{C}$ given by 
\begin{align}
w^\xi_\eta(\rho)
&:=\sum_{\{\alpha,\beta;\epsilon\}\in\vec{\mathcal{\triangle}}^\xi_\eta} 
\left(e^{\frac{1}{2}(\ell_{\alpha}(\rho)+\ell_{\beta}(\rho))}+(-1)^{\alpha\cdot\beta}\right)^{-1}\label{eq:lemtop}\\
&=1-\sum_{\{\alpha,\beta;\epsilon\}\in\vec{\mathcal{\triangle}}^\eta_\xi} 
\left(e^{\frac{1}{2}(\ell_{\alpha}(\rho)+\ell_{\beta}(\rho))}+(-1)^{\alpha\cdot\beta}\right)^{-1}\label{eq:lembot}
\end{align}
is well-defined, holomorphic and gives the complex distance between the (non-$\infty$) endpoint $x$ of $\xi$ and the (non-$\infty$) endpoint $y$ of $\eta$.
\end{lem}

\begin{proof}
Since $w^\xi_\eta$ is a subseries of \eqref{eq:holomorphicseries}, our proof of Proposition~\ref{thm:holomorphicity} ensures that $w^\xi_\eta$ is a well-defined and holomorphic function. To show that $w^\xi_\eta$ satisfies \eqref{eq:lemtop}, \eqref{eq:lembot} and may be interpreted as the complex distance between $x$ and $y$, we show that these properties are satisfied on the Fuchsian locus and invoke the identity theorem.\medskip

We first observe that $x$ and $y$ may be regarded as holomorphic functions on $\mathcal{QF}(N)$ as follows: given an arbitrary element $\rho\in\mathcal{QF}(N)=\mathcal{T}(N)$, let $\mu$ be a Beltrami differential on $dN$ representing $\rho$ and consider the canonical $\mu$-quasiconformal mapping $\psi_\mu:\hat{\mathbb{C}}\rightarrow\hat{\mathbb{C}}$. Even though $\psi_\mu$ is dependent on the representative $\mu$ chosen, the restriction of $\psi_\mu$ to $\hat{\mathbb{R}}$ is independant of $\mu$ as $\psi_\mu$ must take the attracting fixed-points of $\rho_0(\gamma)$ to the corresponding attracting fixed points of $\rho(\gamma)$ for every $\gamma\in\pi_1(N)$. We denote this restricted function by $\psi_\rho$. The holomorphic dependence of $\psi_\mu$ with respect to $\mu$ (see, e.g.: Theorem~4.37 of\cite{imayoshitaniguchi}) ensures that the function $\psi_{(\cdot)}(\cdot):\mathcal{QF}(N)\times\hat{\mathbb{R}}\rightarrow\hat{\mathbb{C}}$ that takes $(\rho,z)$ to $\psi_\rho(z)$ is holomorphic in the first coordinate. Take $x_0$ and $y_0$ to be the points in $\hat{\mathbb{R}}=C_{\rho_0}\cup\{\infty\}$ which constitute the respective non-$\infty$ endpoints of $\xi$ and $\eta$ with respect to $\rho_0$. Then, define the holomorphic functions $x(\rho):=\psi_\rho(x_0)$ and $y(\rho):=\psi_\rho(y_0)$. The function $\omega^\xi_\eta: \mathcal{QF}(N)\rightarrow\mathbb{C}$ defined by
\begin{align}
\omega^\xi_\eta(\rho):= \psi_{\rho}(x_0)-\psi_{\rho}(y_0)=x(\rho)-y(\rho)
\end{align}
is therefore also holomorphic.\medskip

When $\rho$ is in the Fuchsian locus, the number $\omega^\xi_\eta(\rho)$ is equal to the length of the horocyclic segment on the length $1$ horocycle truncated by $\xi$ and $\eta$ (as measured in the direction along $m_p$ from $\xi$ to $\eta$). The Birman-Series theorem tells us that the length of this horocyclic segment is equal to the sum of all of the McShane identity ``gaps" lying on this segment. This is precisely expressed by the following identity as a consequence of the geometric interpretation of the usual Fuchsian identity:
\begin{align}
\omega^\xi_\eta(\rho)
=w^\xi_\eta(\rho)
:=\sum_{\{\alpha,\beta;\epsilon\}\in\vec{\mathcal{\triangle}}^\xi_\eta} 
\left(e^{\frac{1}{2}(\ell_{\alpha}(\rho)+\ell_{\beta}(\rho))}+(-1)^{\alpha\cdot\beta}\right)^{-1}.\label{eq:lemtopplus}
\end{align}
As with the proof of Theorem~\ref{thm:mcshane}, the identity theorem extends the above Fuchsian identity \eqref{eq:lemtopplus} over the entire quasifuchsian character variety. Replacing $\mathcal{S}(N)$ by $\mathcal{S}(N)\times\{\pm\}=\vec{\triangle}=\vec{\triangle}^\xi_\eta\cup\vec{\triangle}^\eta_\xi$ in the expression of Theorem~\ref{thm:mcshane} doubles the $\frac{1}{2}$ on the right-hand side to a $1$, hence giving us equation~\eqref{eq:lembot}. The complex distance interpretation is because $\omega^\xi_\eta=w^\xi_\eta$, and the former is defined to be the complex difference between $x$ and $y$.
\end{proof}

\subsubsection{The horo-core annulus identity}

We now prove Theorem~\ref{thm:coremodulus}.

\begin{proof}
Let $z_-$ and $z_+$ respectively be lowest and highest height points inducing the width partition $(\vec{\triangle}^+(\rho),\vec{\triangle}^-(\rho))$, we first assume that $z_\pm$ correspond (as described in the paragraph before Definition~\ref{dfn:width}) to simple bi-infinite geodesics $\zeta_\pm\in\vec{\mathcal{G}}-\vec{\triangle}$. Then, we may set $\xi=\zeta_+$ (i.e.: $x=z_+$) and $\eta=\zeta_-$ (i.e.: $y=z_-$) in the context of Lemma~\ref{thm:height}, which in turn means that $\vec{\triangle}^\xi_\eta=\vec{\triangle}^+(\rho)$ and $\vec{\triangle}^\eta_\xi=\vec{\triangle}^-(\rho)$. Then, by taking the imaginary component of equations~\eqref{eq:lemtop} and \eqref{eq:lembot}, we obtain that 
\begin{align*}
\mathrm{Im}(w^\xi_\eta(\rho))
= \mathrm{Im}&
\sum_{\{\alpha,\beta;\epsilon\}\in\triangle^+(\rho)}\left(e^{\frac{1}{2}(\ell_{\alpha}(\rho)+\ell_{\beta}(\rho))}+(-1)^{\alpha\cdot\beta}\right)^{-1}\\
=-\mathrm{Im}&
\sum_{\{\alpha,\beta,\epsilon\}\in\triangle^-(\rho)}\left(e^{\frac{1}{2}(\ell_{\alpha}(\rho)+\ell_{\beta}(\rho))}+(-1)^{\alpha\cdot\beta}\right)^{-1}.
\end{align*}
To show that $\mathrm{Im}(w^\xi_\eta(\rho))=\mathrm{mod}_p(\rho)$, observe that the horo-core annulus is bounded above by the two hyperbolic planes $Pl_\pm\subset\mathbb{H}^3$ with respective ideal boundaries given by
\[
\{u+iv \mid v=\mathrm{Im}(z_\pm)\}\cup\{\infty\}\subset\hat{\mathbb{C}}.
\]
Thus, it is conformally equivalent to a flat annulus obtained by gluing a rectangle of length $1$ and width $\mathrm{Im}(z_+)-\mathrm{Im}(z_-)=\mathrm{Im} (w^\xi_\eta(\rho))$. It is well-known that this width also the modulus $\mathrm{mod}_p(\rho)$ of this flat annulus.\medskip

So far, we have established the result in the case when $\zeta_\pm\in\vec{\mathcal{G}}-\vec{\triangle}$. To complete our proof, we consider the case when at least one of $\zeta_\pm$ is either self-intersecting or in $\vec{\triangle}$ and show that it is possible to replace them with simple bi-infinite geodesics which spiral to simple closed geodesics (i.e.: elements of $\vec{\mathcal{G}}-\vec{\triangle}$).\medskip

Let us assume without loss of generality that $\zeta_+\notin\vec{\mathcal{G}}-\vec{\triangle}$. Since $z_+$ is a highest point on the limit curve $C_\rho$, the geodesics $\zeta_+$ must lie on the boundary of the convex core. This in turn means that it must not (transversely) intersect the pleating locus. If not, curve shortening near the pleating locus would show that there is a curve homotopy equivalent to, but locally shorter than, the geodesic $\zeta_+$. Thus, $\zeta_+$ lies on a geodesic-bordered (smooth) hyperbolic subsurface $X_+$ within the top boundary of the convex core. In particular, the fattening of any sufficiently small $\epsilon$-neighborhood of the subsegment of $\zeta_+$ up to its first point of self-intersection  (on the convex core boundary) is topologically a pair of pants (it cannot be a $1$-holed M\"obius band because the pleated geodesic boundary of the convex core of $\mathbb{H}^3/\rho(\pi_1(N))$ is topologically equivalent to an orientable surface $dN$). Since $X_+$ is geodesically convex, it must therefore contain a geodesic bordered pair of pants which contains $\zeta_+$ up to its first point of self-intersection. Since $\vec{\mathcal{G}}-\vec{\triangle}$ is a Cantor set (Corollary~\ref{thm:cantor}), this means that $\zeta_+$ is launched between within a gap region bounded by simple bi-infinite geodesics $\nu_1$ and $\nu_2$ (lying on $X_+$) which spiral to simple closed geodesics (Theorem~\ref{thm:simplecharacterization}).\medskip

It should be noted that $Pl_+$ contains a lift of the universal cover of $X_+$, and therefore lifts of $\nu_1,\nu_2$ emanating from $\infty$ must have the same height as a lift of $\zeta_+$ emanating from $\infty$. This means that the complex distance between the $\nu_i$ is strictly real, and replacing $\zeta_+$ with $\nu_1$ (or $\nu_2$) does not affect equations~\eqref{eq:modtop} and \eqref{eq:modbot}. Therefore, we may assume without loss of generality that $\zeta_\pm\in\vec{\mathcal{G}}-\vec{\triangle}$, as desired.
\end{proof}

\section{Identities for pseudo-Anosov mapping Klein bottles}
\label{sec:final}

The goal of this section is to prove the McShane identity (Theorem~\ref{thm:zero}) for pseudo-Anosov mapping Klein bottles, as well as to use these summands to describe the cusp geometry of any given pseudo-Anosov mapping Klein bottle $K_{(d\varphi,\iota)}$. Recall from \S\ref{sec:pseudo} that the mapping torus $M_{d\varphi}$ is a double cover of $K_{(d\varphi,\iota)}$. We shall make use of the interplay between these two hyperbolic $3$-manifolds. To begin with, let us clarify some notation. Given cusp $p$ on $N$, there are two cusps on $dN$ which cover $p$ and we denote them by $q$ and $q'$.

\subsection{The statement of the cuspidal tori identity}

Any embedded horospheric cross-section of cusp $p$ in $K_{(d\varphi,\iota)}$ is the same Euclidean torus $T_{(d\varphi,\iota)}$ up to homothety. Given a pair of generators $[\alpha],[\beta]$ for $\pi_1(T_{(d\varphi,\iota)})$ (i.e.: a marking on $T_{(d\varphi,\iota)}$), we define the \emph{marked modulus} of $(T_{(d\varphi,\iota)}, \{[\alpha],[\beta]\})$ to be the Teichm\"uller space parameter for this marked torus in the Teichm\"uller space $\mathit{Teich}_{1,0}=\mathbb{H}^2\subset\mathbb{C}$ (see, for example, \S1.2.2 of\cite{imayoshitaniguchi}). The cusp torus $T_{(d\varphi,\iota)}\subset K_{(d\varphi,\iota)}$ at $p$ lifts to two distinct cusp tori in $dN$, with one based at $q$ and the other at $q'$. We shall at times study $T_{(d\varphi,\iota)}$ via its lift at the cusp $p$ in $dN$.


\begin{figure}[h!]
\centering
\includegraphics[width=0.8\textwidth]{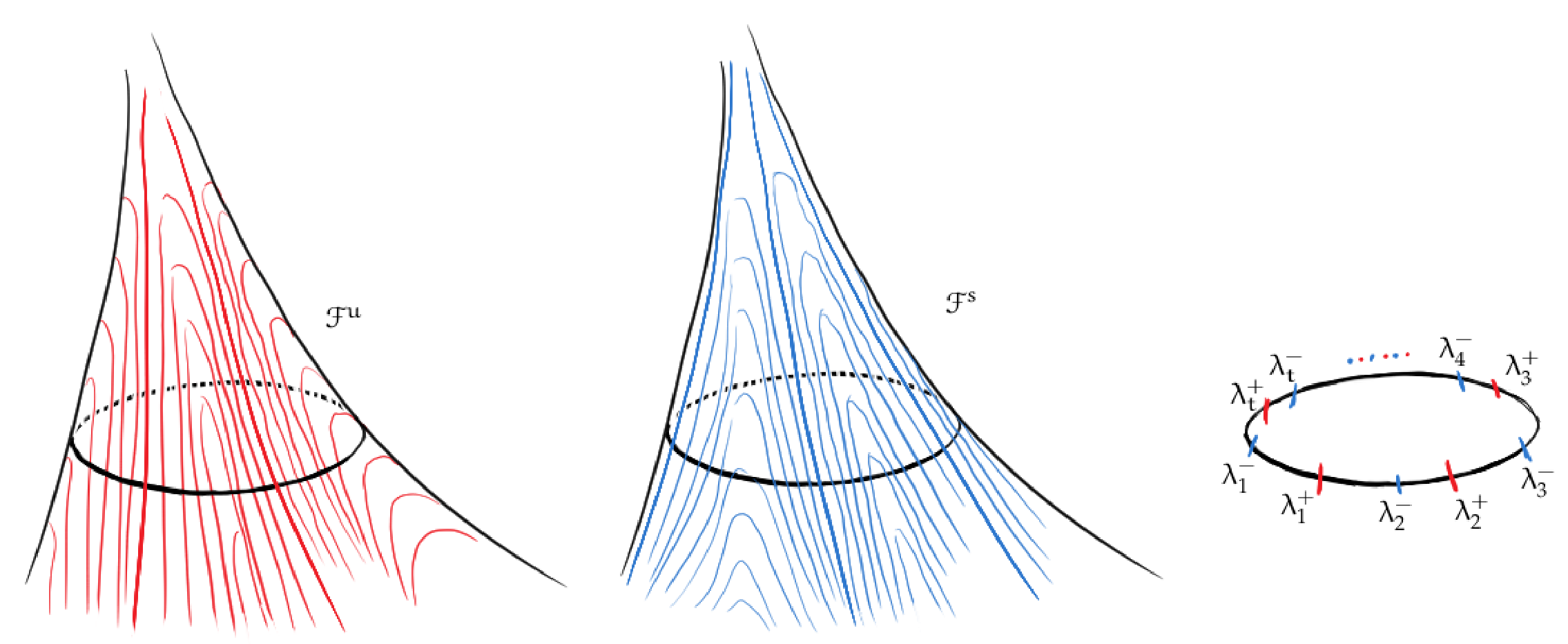}
\caption{(left to right) the stable foliation $\mathcal{F}^s$ around cusp $q$; the unstable foliation $\mathcal{F}^u$ around cusp $q$; singular foliations as points in the set $\mathbb{S}^1_q$ of directions emanating from $q$ on $dN$.}\label{fig:stableunstable}
\end{figure}

\begin{dfn}[Signature of an orientation-preserving pseudo-Anosov map]
Given an orientation-preserving pseudo-Anosov map $d\varphi:dN\to dN$, the singular leaves $\{\lambda^+_1,\ldots,\lambda^+_t\}$ of its stable foliation around cusp $q$ in $dN$ and the singular leaves $\{\lambda^-_1,\ldots,\lambda^-_t\}$ of its unstable foliation around cusp $q$ in $dN$ interlace one another as illustrated in Figure~\ref{fig:stableunstable}. The pseudo-Anosov map $d\varphi$ preserves each set of singular leaves and acts on $\{\lambda^\pm_1,\ldots,\lambda^\pm_t\}$ by cyclic permutation, shifting the index by some $s\in\{0,\ldots,t-1\}$. We refer to the pair $(s,t)$ as the \emph{signature} of the pseudo-Anosov map $d\varphi$ at $q$. When $s=0$, we say that $d\varphi$ has \emph{simple signature} at $q$.
\end{dfn}

Given any twisted-pA pair $(d\varphi,\iota)$, where $d\varphi:dN\to dN$ has simple signature at $q$ (and hence at $q'$). There is a canonical marking on the cusp torus $T_{(d\varphi,\iota)}$ at $p$ by taking the pair $(m_p,l_p)$, where the meridian $m_p$ is a loop around cusp $p$ and the longitude $l_p$ may be constructed as follows:

\begin{dfn}[longitude]\label{dfn:longitude}
Recall that $K_{(d\varphi,\iota)}$ may be constructed by taking $dN\times [0,\frac{1}{2}]$ and identifying $dN\times\{0\}$ via $\iota$ and $dN\times\{\frac{1}{2}\}$ via $\varphi\circ\iota$. Choose an arbitrary point $x$ lying on an arbitrary singular leaf $\lambda$ at $q$ on $dN\times\{t\neq0,\frac{1}{2}\}$, the intervals $\{x\}\times[0,\frac{1}{2}]$ and $\{\iota(x)\}\times[0,\frac{1}{2}]$ join up to form a path on $K_{(d\varphi,\iota)}$ because $(x,0)\sim(\iota(x),0)$. Moreover, it makes sense to assert that both end points $(x,\frac{1}{2})$ and $(\iota(x),\frac{1}{2})$ lie on the same singular leaf on $N\times\{\frac{1}{2}\}$ because:
\begin{itemize}
\item
$d\varphi\circ\iota$ exchanges the stable and unstable foliations on $N$ and hence the foliation structure decends to $N\times\{\frac{1}{2}\}=dN/d\varphi\circ\iota$;
\item
since $d\varphi\circ\iota=\iota\circ(d\varphi)^{-1}$, hence $(\iota(x),\frac{1}{2})=(\iota\circ(d\varphi)^{-1}\circ(d\varphi)(x),\frac{1}{2})\sim(d\varphi(x),\frac{1}{2})$;
\item
$d\varphi$ having simple signature means that $(d\varphi(x),\frac{1}{2})$ and $(x,\frac{1}{2})$ descend to the same singular leaf on $N\times\{\frac{1}{2}\}$.
\end{itemize}
Joining these end points along the singular leaf then results in a simple closed loop $l_p$ that is, up to homotopy, independent of our choice of $\lambda$ and $x$. We call $l_p$ the \emph{longitude} of the cusp torus $T_{(d\varphi,\iota)}\subset K_{(d\varphi,\iota)}$ at cusp $p$.
\end{dfn}

\begin{note}
The lifts $l_q,l_{q'}$ of the longitude $l_p$ to either cusp torus in $dN$ agrees with the notion of longitude given in Definition~3.4 of\cite{amsgroups} for cuspidal tori on pseudo-Anosov mapping tori.
\end{note} 
 
The singular leaves $\lambda^\pm_i$ are singular only at cusp $q$, and thus form simple bi-infinite paths on $dN$. The geodesic representative for $\lambda^+_i$ (resp. $\lambda^-_i$) has one end at cusp $q$ and the other end spirals towards a leaf of the unstable (resp. stable) measured lamination of $d\varphi$, and we endow each $\lambda^\pm_i$ with the orientation going from the cusp $q$ to the measured lamination. We regard the cyclically ordered set $\{\lambda^-_1,\lambda^+_1,\ldots,\lambda^-_t,\lambda^+_t\}$ of interlacing singular leaves as a cyclically ordered set of ``directions" in the circle's worth of ``directions" $\mathbb{S}^1_q$ emanating from cusp $q$ (see Figure~\ref{fig:stableunstable}) on $dN$. We use these singular leaves on $dN$ to partition $\triangle$, the set of unoriented ideal geodesics on $N$ with both ends at $p$ (see Definition~\ref{defn:sidedness}) via a partitioning algorithm.\medskip

Every unoriented ideal geodesic $\sigma\in\triangle$ on $N$ is covered by $4$ \emph{oriented} ideal geodesics on $dN$, of which precisely $2$ have their sources at $q$. We refer to these two oriented ideal geodesics as the \emph{$q$-source lifts} of $\sigma$.

\begin{itemize}
\item 
$\triangle^+_q$: the set of ideal geodesics $\sigma\in\triangle$ where both $p$-source lifts of $\sigma$ are launched within an interval of the form $(\lambda^-_i,\lambda^+_i)\subset\mathbb{S}^1_q$;

\item 
$\triangle^-_q$: the set of ideal geodesics $\sigma$ where both $p$-source lifts of $\sigma$ are launched within an interval of the form $(\lambda^+_i,\lambda^-_{i+1})\subset\mathbb{S}^1_q$;

\item
$\triangle^0_q$: the remaining set of ideal geodesics consisting of those with one $q$-source lift launched within each of the two interval types.
\end{itemize}
Proposition~\ref{thm:bijection} translates the above partition of $\triangle$ into the partition $\mathcal{S}(N)=\mathcal{S}^+(N)\sqcup\mathcal{S}^0(N)\sqcup\mathcal{S}^-(N)$. Since $d\varphi$ fixes the singular leaves $\{\lambda^\pm_i\}$, this partition is $\varphi$-invariant and descends to a partition 
\[
\mathcal{S}_{(d\varphi,\iota)}=\mathcal{S}^+_{(d\varphi,\iota)}\sqcup\mathcal{S}^0_{(d\varphi,\iota)}\sqcup\mathcal{S}^-_{(d\varphi,\iota)}
\]
of the collection of $K_{(d\varphi,\iota)}$-homotopy classes of pairs of pants on either of the two exceptional $N$ fibers of $K_{(d\varphi,\iota)}$.

\begin{note}\label{note:psourcebijection}
Since there are two $q$-source lifts for each ideal ideal geodesics in $\triangle$, and each such lift emanates from $q$ in a different direction, the set of all $q$-source lifts of ideal geodesics in $\triangle$ naturally identifies with $\vec{\triangle}$ --- the set of oriented simple ideal geodesics on $N$ with both ends at $p$. This identification by no means a coincidence and comes from the simple fact that the set of directions  $\mathbb{S}^1_p$ emanating from $p$ naturally agree with the set of directions $\mathbb{S}^1_q$ emanating from $q$. This natural bijection of directions is useful for the proof of Theorem~\ref{thm:kleinbottle}, where we will be working with ideal geodesics on $dN$ emanating from $p$ but gather corresponding summands indexed by ideal geodesics in $\vec{\triangle}$ which are emanate from $p$ on $N$.
\end{note}

\begin{thm}
\label{thm:kleinbottle}
Given a twisted-pA pair $(d\varphi,\iota)$ where the pseudo-Anosov map $\varphi$ has simple signature $(0,t)$, the \emph{marked modulus} $\mathrm{mod}_p(d\varphi,\iota)$, with respect to the marking $(m_p,l_p)$, of the cusp-$p$ torus $T_{(d\varphi,\iota)}$ of the pseudo-Anosov mapping Klein bottle $K_{(d\varphi,\iota)}$ is given by:
\begin{align}
\mathrm{mod}_p(d\varphi,\iota)
=&\left(\frac{2}{t}\sum_{\{\alpha,\beta\}\in\mathcal{S}^+_{(d\varphi,\iota)}}+\frac{1}{t}\sum_{\{\alpha,\beta\}\in\mathcal{S}^0_{(d\varphi,\iota)}}\right)
\left(e^{\frac{1}{2}(\ell_{\alpha}(\phi)+\ell_{\beta}(\phi))}+(-1)^{\alpha\cdot\beta}\right)^{-1},\label{eq:mappingtori}\\
=-&\left(\frac{2}{t}\sum_{\{\alpha,\beta\}\in\mathcal{S}^-_{(d\varphi,\iota)}}+\frac{1}{t}\sum_{\{\alpha,\beta\}\in\mathcal{S}^0_{(d\varphi,\iota)}}\right)
\left(e^{\frac{1}{2}(\ell_{\alpha}(\phi)+\ell_{\beta}(\phi))}+(-1)^{\alpha\cdot\beta}\right)^{-1}.\label{eq:mappingtorineg}
\end{align}
\end{thm}

\begin{note}
Theorem~\ref{thm:zero} for the special case that $d\varphi$ has simple signature is an immediate corollary of Theorem~\ref{thm:kleinbottle}.
\end{note}

For a pseudo-Anosov map $d\varphi$ with general signature $(s,t)$, the pseudo-Anosov map $\hat{d\varphi}:=(d\varphi)^{\frac{t}{\mathrm{gcd}(s,t)}}$ and all of its integer powers have simple signature at $q$. We use this in \S\ref{sec:generalsignature} to extend Theorem~\ref{thm:kleinbottle} to the general signature case.

\subsection{Proof for the simple signature case}
Consider a twisted-pA pair $(d\varphi,\iota)$ where $\varphi$ is a pseudo-Anosov homeomorphism of simple signature. Instead of working with the cusp $p$ torus $T_{(d\varphi,\iota)}$ on the pseudo-Anosov Klein bottle $K_{(d\varphi,\iota)}$, we shall work with the cusp $q$ torus on the pseudo-Anosov mapping torus $M_{d\varphi}$ which doble-covers $K_{(d\varphi,\iota)}$. Let $\phi:\pi_1(M_\varphi)\rightarrow\mathrm{PSL}(2,\mathbb{C})$ denote its holonomy representation. The longitude $l_q$ of $d\varphi$ is a candidate for the stable letter $l$ for the fundamental group $\pi_1(M_{d\varphi})$ as a HNN-extension. This means that the meridian $m_q$ and the longitude $l_q$ define a canonical $\mathbb{Z}$-basis $(m_q,l_q)$ for the fundamental group of the cusp torus at $q$. We use this basis as a marking basis for the cusp $q$ lift of the cusp $p$ torus $T_{(d\varphi,\iota)}$. 


\begin{proof}
Given the pseudo-Anosov map $d\varphi$, there is an associated collection of oriented simple geodesics
\[
\{\lambda_1^-,\lambda_1^+,\lambda_2^-,\lambda_2^+,\ldots,\lambda_t^-,\lambda_t^+\}\subset\mathbb{S}^1_q
\]
consisting of geodesic representatives for the singular leaves, at $q$, of the stable and unstable foliations of $d\varphi$. We fix a (Cantor set) boundary point $\xi_i^+\in(\lambda_i^-,\lambda_i^+)\cap\vec{\mathcal{G}}$ and  a boundary point $\xi_i^-\in(\lambda_i^+,\lambda_{i+1}^-)\cap\vec{\mathcal{G}}$ for each $i$. Since $\xi_i^-$ (resp. $\xi_i^+$) is a boundary point of $\vec{\mathcal{G}}$, the underlying oriented simple geodesic spirals to some oriented simple closed $2$-side geodesic on $dN$, which we denote by $\gamma_i^-$ (resp. $\gamma_i^+$).\medskip

Let $\mathrm{Fix}^+(A)$ denote the attracting fixed point of a loxodromic M\"obius transformation $A\in\mathrm{PSL}(2,\mathbb{C})$. Since $\phi(l_q)$ acts on $\mathbb{C}=\partial_\infty\mathbb{H}^3-\{\infty\}$ via translation, it is explicitly expressed as an addition by some complex number $\mathrm{mod}_q(d\varphi)$ and for $\gamma=\gamma^\pm_i$ we have: 
\begin{align*}
\mathrm{mod}_q(d\varphi)
&=\phi(l)\cdot\mathrm{Fix}^+(\phi(\gamma))-\mathrm{Fix}^+(\phi(\gamma))\\
&=\mathrm{Fix}^+(\phi(l\gamma l^{-1}))-\mathrm{Fix}^+(\phi(\gamma))\\
&=\mathrm{Fix}^+(\phi(d\varphi_*\gamma))-\mathrm{Fix}^+(\phi(\gamma)).
\end{align*}

Furthermore, the restriction of $\phi$ to $\pi_1(dN)\leq \pi_1(M_\varphi)$ is the strong limit of a path $\{\rho_\tau\}$ of quasifuchsian representations of $\pi_1(dN)$, therefore 
\begin{align*}
\mathrm{mod}_q(\varphi)
=\lim_{\tau\to\infty}
\left(\mathrm{Fix}^+(\rho_\tau(d\varphi_*\gamma))-\mathrm{Fix}^+(\rho_\tau(\gamma))\right).
\end{align*}
By construction, we know that $\eta^+_i:=d\varphi_*\xi_i^+$ comes after $\xi^+_i$ on the interval $(\lambda^-_i,\lambda^+_i)$. Identifying $\mathbb{S}^1_p$ and $\mathbb{S}^1_q$ as per Note~\ref{note:psourcebijection} then lets us view $\eta^+_i$ and $\xi^+_i$ as elements of $\mathbb{S}^1_p$ and Lemma~\ref{thm:height} then tells us that:  
\begin{align*}
w^{\xi^+_i}_{\eta^+_i}(\rho_\tau)
&=\sum_{[\{\alpha,\beta;\epsilon\}]\in\vec{\triangle}^{\xi^+_i}_{\eta^+_i}}
\left(e^{\frac{1}{2}(\ell_{\alpha}(\rho_\tau)+\ell_{\beta}(\rho_\tau))}+(-1)^{\alpha\cdot\beta}\right)^{-1}.
\end{align*}

Since $\mathrm{Fix}^+(\rho_\tau(d\varphi_*\gamma_i^+))$ and $\mathrm{Fix}^+(\rho_\tau(\gamma_i^+))$ are  the respective non-$\infty$ end-points for $\eta_i^+:=d\varphi_*\xi_i^+$ and $\xi_i^+$, the series $w^{\xi^+_i}_{\eta^+_i}(\rho_\tau)
$ is precisely given by $\mathrm{Fix}^+(\rho_\tau(d\varphi_*\gamma_i^+))-\mathrm{Fix}^+(\rho_\tau(\gamma_i^+))$ and hence
\begin{align}
\mathrm{mod}_q(d\varphi)
=\lim_{\tau\to\infty}\sum_{[\{\alpha,\beta;\epsilon\}]\in\vec{\triangle}^{\xi^+_i}_{\eta^+_i}}
\left(e^{\frac{1}{2}(\ell_{\alpha}(\rho_\tau)+\ell_{\beta}(\rho_\tau))}+(-1)^{\alpha\cdot\beta}\right)^{-1}.\label{eq:tempseries}
\end{align}

On the other hand, we know by construction that $\eta^-_i:=d\varphi_*\xi_i^-$ comes before $\xi^-_i$ on $(\lambda^+_i,\lambda^-_{i+1})$ and so:
\begin{align*}
w^{\eta^-_i}_{\xi^-_i}(\rho_\tau)
=\sum_{[\{\alpha,\beta;\epsilon\}]\in\vec{\triangle}^{\eta^-_i}_{\xi^-_i}}
\left(e^{\frac{1}{2}(\ell_{\alpha}(\rho_\tau)+\ell_{\beta}(\rho_\tau))}+(-1)^{\alpha\cdot\beta}\right)^{-1}.
\end{align*}
This time, the width $w^{\eta^-_i}_{\xi^-_i}(\rho_\tau)$ is equal to $\mathrm{Fix}^+(\rho_\tau(\gamma^-_i))-\mathrm{Fix}^+(\rho_\tau(\varphi_*\gamma^-_i))$, and we instead obtain:

\begin{align}
\mathrm{mod}_q(d\varphi)
=-\lim_{\tau\to\infty}\sum_{[\{\alpha,\beta;\epsilon\}]\in\vec{\triangle}^{\eta^-_i}_{\xi^-_i}}
\left(e^{\frac{1}{2}(\ell_{\alpha}(\rho_\tau)+\ell_{\beta}(\rho_\tau))}+(-1)^{\alpha\cdot\beta}\right)^{-1}.\label{eq:tempseriesneg}
\end{align}

We now turn to the summation index sets $\vec{\triangle}^{\xi^+_1}_{\eta^+_1},\ldots,\vec{\triangle}^{\xi^+_t}_{\eta^+_t}$ and $\vec{\triangle}^{\eta^-_1}_{\xi^-_1},\ldots,\vec{\triangle}^{\eta^-_t}_{\xi^-_t}$. Since $\lambda^+_i$ are attractive fixed points of the action of $d\varphi$ on $\mathbb{S}^1_q\equiv\mathbb{S}^1_p$ and $\lambda^-_i$ are the repelling fixed points, the interval $[\xi^+_i,\eta^+_i)$ is a fundamental domain for the action of $d\varphi$ on $(\lambda^-_i,\lambda^+_i)$. This in turn means that $d\varphi_*$ induces a bijection between $\vec{\triangle}^{\xi^+_i}_{\eta^+_i}$ and
\[
\left((\lambda^-_i,\lambda^+_i)\cap\vec{\triangle}\right)/d\varphi_*(\zeta)\sim\zeta,
\]
where $\vec{\triangle}\subset\mathbb{S}^1_p$ is regarded as a subset of $\mathbb{S}^1_q$ when it comes to the $d\varphi_*$ quotient (see Note~\ref{note:psourcebijection}). Likewise, we get a bijection between $\vec{\triangle}^{\eta^-_i}_{\xi^-_i}$ and $\left((\lambda^+_i,\lambda^-_{i+1})\cap\vec{\triangle}\right)/d\varphi_*$ and hence the following bijection:
\begin{align}
\vec{\triangle}^{\xi^+_1}_{\eta^+_1}\cup\ldots\cup\vec{\triangle}^{\xi^+_t}_{\eta^+_t}
\cup
\vec{\triangle}^{\eta^-_1}_{\xi^-_1}\cup\ldots\cup\vec{\triangle}^{\eta^-_t}_{\xi^-_t}
\equiv\left(\vec{\triangle}-\{\lambda^\pm_i\}\right)/d\varphi_*
\equiv\vec{\triangle}/d\varphi_*.\label{eq:quotientbijection}
\end{align}
The latter equivalence in \eqref{eq:quotientbijection} utilizes the fact that the stable and unstable leaves $\lambda^\pm_i$ cannot have both ends up $q$. This can be demonstrated by contradiction: the fattened pair of pants or M\"obius band of a stable or an unstable leaf $\lambda$ must be (topologically) fixed under the homeomorphic action of $d\varphi$ (see Note~\ref{note:equivariance}), this in turn means that $d\varphi$ preserves the homotopy class of one of the simple closed geodesic boundaries of the fattening of $\lambda$. This is impossible for a pseudo-Anosov map $d\varphi$ according to the classification of surface homeomorphisms.\medskip

By Note~\ref{note:equivariance} and Note~\ref{note:psourcebijection}, we know that $\vec{\triangle}/d\varphi_*$ naturally identifies with $\mathcal{S}_{(d\varphi,\iota)}\times\{\pm\}$, where $\pm$ are arbitary assignments of orientation. Thus, by summing \eqref{eq:tempseries} and \eqref{eq:tempseriesneg} over $i$, replacing the indices and invoking Proposition~7.6 of\cite{amsgroups} to ensure term-by-term convergence as $\rho_\tau$ tends to $\phi$, we obtain:
\begin{align*}
\sum_{\{\alpha,\beta;\epsilon\}\in\mathcal{S}_{(d\varphi,\iota)}\times\{\pm\}}
\left(e^{\frac{1}{2}(\ell_{\alpha}(\phi)+\ell_{\beta}(\phi))}+(-1)^{\alpha\cdot\beta}\right)^{-1}
=t\;\mathrm{mod}_q(d\varphi)-t\;\mathrm{mod}_q(d\varphi)=0.
\end{align*}
Since the actual summands are independant of $\epsilon\in\{\pm\}$, we may halve the above expression and replace the index set by $\mathcal{S}_{(d\varphi,\iota)}$. Note that this suffices to prove Theorem~\ref{thm:zero} when $d\varphi$ has simple signature.\medskip

Instead of summing over $\mathcal{S}_{(d\varphi,\iota)}\times\{\pm\}$, we may instead sum only over $\vec{\triangle}^{\xi^+_1}_{\eta^+_1}\cup\ldots\cup\vec{\triangle}^{\xi^+_t}_{\eta^+_t}$. This is equivalent to summing over the collection of all oriented ideal geodesics $\zeta\in\vec{\triangle}/d\varphi_*$ which shoot out from $p$ within some interval $\left(\vec{\triangle}\cap\bigcup_{i=1}^t(\lambda^-_i,\lambda^+_i)\right)/d\varphi_*$. This is tantamount to summing over $\triangle^+_q/d\varphi_*$ (and hence $\mathcal{S}_{(d\varphi,\iota)}^+$) twice and $\triangle^0_q$ (and hence $\mathcal{S}_{(d\varphi,\iota)}^0$) once, and yields
\begin{align*}
t\cdot\mathrm{mod}_q(d\varphi)
=\left(2\sum_{\{\alpha,\beta\}\in\mathcal{S}^+_{(d\varphi,\iota)}}+\sum_{\{\alpha,\beta\}\in\mathcal{S}^0_{(d\varphi,\iota)}}\right)
\left(e^{\frac{1}{2}(\ell_{\alpha}(\phi)+\ell_{\beta}(\phi))}+(-1)^{\alpha\cdot\beta}\right)^{-1},
\end{align*}
which in turn gives us \eqref{eq:mappingtori} as desired. Equation~\eqref{eq:mappingtorineg} is either similarly derived by summing over $\vec{\triangle}^{\eta^-_1}_{\xi^-_1}\cup\ldots\cup\vec{\triangle}^{\eta^-_t}_{\xi^-_t}$ or by applying Theorem~\ref{thm:zero}.\medskip

Finally, since the marking generators $(m_p,l_p)$ for the cusp torus at $p$ are respectively sent to $\pm\left[\begin{smallmatrix}1&1\\0&1\end{smallmatrix}\right]$ and $\pm\left[\begin{smallmatrix}1& \mathrm{mod}_q(d\varphi)\\0&1\end{smallmatrix}\right]$. The marked modulus for the cusp $q$ lift of $T_{(d\varphi,\iota)}$ with respect to the marking generator set $(m_q,l_q)$ is $\mathrm{mod}_q(d\varphi)$. But by the since the two Euclidean tori are conformally equivalent, the marked modulus for $T_{(d\varphi,\iota)}$ with respect to the marking $(m_p,l_p)$ is $\mathrm{mod}_p(d\varphi,\iota)=\mathrm{mod}_q(d\varphi)$ as desired.
\end{proof}

\subsection{Identities for the general signature case}
\label{sec:generalsignature}

We conclude this section by addressing what happens when the pseudo-Anosov map $d\varphi$ in a twisted-pA pair $(d\varphi,\iota)$ has general signature. Assume that $d\varphi: dN\rightarrow dN$ has signature $(s,t)$, then the map
\[
\hat{d\varphi}:=(d\varphi)^{\frac{t}{\mathrm{gcd}(s,t)}}:dN\rightarrow dN
\] 
is also pseudo-Anosov but of simple signature $(0,t)$. Then  $M_{\hat{d\varphi}}$ is an order $\frac{t}{\mathrm{gcd}(s,t)}$ finite cover of $M_{d\varphi}$ via a covering map $\Pi:M_{\hat{d\varphi}}\rightarrow M_{d\varphi}$. We denote the respective holonomy representations for these the two pseudo-Anosov mapping tori $M_{d\varphi}$ and $M_{\hat{d\varphi}}$ by $\phi$ and $\hat{\phi}$. Let us now prove Theorem~\ref{thm:zero}.

\begin{proof}[Proof of Theorem~\ref{thm:zero}]
We first observe that $(\hat{d\varphi},\iota)$ is in fact a twisted-pA pair, and its corresponding pseudo-Anosov mapping Klein bottle is a order $2$ quotient of the pseudo-Anosov mapping torus $M_{\hat{d\varphi}}$. Since $\hat{d\varphi}$ has signature, the proof of Theorem~\ref{thm:kleinbottle} tells us that
\begin{align}
\sum_{\{\hat{\alpha},\hat{\beta}\}\in\mathcal{S}_{(\hat{d\varphi},\iota)}}
\left(e^{\frac{1}{2}(\ell_{\hat{\alpha}}(\hat{\phi})+\ell_{\hat{\beta}}(\hat{\phi}))}+(-1)^{\hat{\alpha}\cdot\hat{\beta}}\right)^{-1}=0.\label{eq:finitequotient}
\end{align}
Since $M_{\hat{d\varphi}}$ is an order $\frac{t}{\mathrm{gcd}(s,t)}$ finite cover of $M_{d\varphi}$, for each pair of geodesics $\{\alpha,\beta\}$ corresponding to a pair of pants or a $1$-holed M\"obius band in $\mathcal{S}_{(d\varphi,\iota)}$, there are $\frac{t}{\mathrm{gcd}(s,t)}$ isometrically configured pairs of geodesics $\{\hat{\alpha},\hat{\beta}\}$ covering it in $\mathcal{S}_{(\hat{d\varphi},\iota)}$. This means that we may simply divide \eqref{eq:finitequotient} by $\frac{t}{\mathrm{gcd}(s,t)}$ and replace $\hat{\varphi},\hat{\alpha},\hat{\beta}$ and $\hat{\phi}$ with $\varphi,\alpha,\beta$ and $\phi$ to obtain the desired result.
\end{proof}

Finally, we turn to the geometry of the cusp $p$ torus $T_{(d\varphi,\iota)}$ on $K_{(d\varphi,\iota)}$, or equivalently the cusp $q$ torus on $M_{d\varphi}$. Unlike the simple signature case, the longitude $l_p$ (or equivalently, $l_q$) does not pair with the meridian $m_p$ (or equivalently, $m_q$) to give a $\mathbb{Z}$-basis for $\pi_1(T_{(d\varphi,\iota)})$. Hence, we instead choose a $\mathbb{Z}$-basis $(m_q,l)$ for the cusp $q$ torus of $M_{d\varphi}$ covering $T_{(d\varphi,\iota)}$, and let $(\hat{m}_q,\hat{l}_q)$ denote the (meridian and longitude) marking generators for the cusp $q$ torus of $M_{\hat{d\varphi}}$. Since $(m_q,l)$ is a $\mathbb{Z}$-basis, there is a unique integer $k_l$ so that 
\begin{align}
\Pi_*(\hat{l}_q)=\tfrac{t}{\mathrm{gcd}(s,t)}\cdot l+k_l\cdot m_q\label{eq:basis}
\end{align}
as homotopy classes in the fundamental group of the cusp $q$ torus of $M_{d\varphi}$. Since the cusp $q$ torus covering $T_{(d\varphi,\iota)}$ is identical to $T_{(d\varphi,\iota)}$, we abuse notation slightly and use $(m_p,l)$ to denote the homotopy group generators for $\pi_1(T_{(d\varphi,\iota)})$ corresponding to $(m_q,l)$.

\begin{cor}
The marked modulus $\mathrm{mod}_p(d\varphi,\iota;l)$ for the cusp $p$ torus $T_{(d\varphi,\iota)}$ of $K_{(d\varphi,\iota)}$, with respect to the basis $(m_p,l)$, is given by:
\begin{align}
\mathrm{mod}_p(d\varphi,\iota;l)
=\left(\frac{2}{t}\sum_{\{\alpha,\beta\}\in\mathcal{S}^+_{(d\varphi,\iota)}}
+\frac{1}{t}\sum_{\{\alpha,\beta\}\in\mathcal{S}^0_{(d\varphi,\iota)}}\right)
\left(e^{\frac{1}{2}(\ell_{\alpha}(\phi)+\ell_{\beta}(\phi))}+(-1)^{\alpha\cdot\beta}\right)^{-1}- \frac{k_l\cdot\mathrm{gcd}(s,t)}{t}\label{eq:mappingtorigeneral}.
\end{align}
\end{cor}

\begin{proof}
For the remainder of this proof, we shall simply work with $\mathrm{mod}_q(d\varphi;l)$ instead of $\mathrm{mod}_p(d\varphi,\iota;l)$, the former being the modulus of the correspondingly marked cusp $q$ torus covering $T_{(d\varphi,\iota)}$. Thanks to our normalization condition that 
\[
\phi(m_q)=\hat{\phi}(\hat{m}_q)=\pm\left[\begin{smallmatrix}1&1\\0&1\end{smallmatrix}\right],
\]
we know that $\phi(l)$ and $\hat{\phi}(\hat{l}_q)$ respectively act on $\mathbb{C}=\partial_\infty\mathbb{H}^3-\{\infty\}$ as translation by $\mathrm{mod}_q(d\varphi;l)$ and $\mathrm{mod}_q(\hat{d\varphi})$. Coupling this with \eqref{eq:basis}, we obtain that:
\begin{align}
\mathrm{mod}_q(\hat{d\varphi})=\tfrac{t}{\mathrm{gcd}(s,t)}\cdot \mathrm{mod}_q(d\varphi;l)+k_l.\label{eq:rejig}
\end{align}
Rearranging equation~\eqref{eq:rejig} to make $\mathrm{mod}_q(d\varphi;l)$ the subject, invoking Theorem~\ref{thm:kleinbottle} to replace $\mathrm{mod}_q(\hat{d\varphi})$ and employing the same index replacement trick as used in the proof of Theorem~\ref{thm:zero} then yields \eqref{eq:mappingtorigeneral}.
\end{proof}

\section{Identities for $\mathrm{Aut}(\mathbb{H}^3)=O^+(1,3)$-representations}

The orthochronous Lorentz group $O^+(1,3)$ is the disjoint double cover of $SO^+(1,3)$ and $P\cdot SO^+(1,3)$, where $P$ is the diagonal matrix $\mathrm{diag}(1,-1,-1,-1)$ with determinant $-1$. This obviously enables us to endow $O^+(1,3)$ with a complex structure, but there is a choice, because one can assign to each component either the complex structure of $SO^+(1,3)=\mathrm{PSL}(2,\mathbb{C})$ or its complex conjugate. We choose to endow $O^+(1,3)$ with the complex structure such that it acts complex analytically on $\hat{\mathbb{C}}=\partial_\infty\mathbb{H}^3$. This means that the $SO^+(1,3)$ component is simply identified with $\mathrm{PSL}(2,\mathbb{C})$ and inherits the ``standard'' complex structure, whereas the $P\cdot SO^+(1,3)$ component is endowed with the conjugate complex structure because $P$ acts anti-holomorphically on $\mathbb{S}^2=\hat{\mathbb{C}}$ as the antipodal map.

 is a complex manifold, this means that $O^+(1,3)$ is also a complex manifold, and this in turn allows the transflected-quasifuchisan space $\mathcal{TQF}(S)$ to inherit a complex structure.

\subsection{Identity for transflected-quasifuchsian representations}

\subsection{Identity for pseudo-Anosov mapping tori}

\section*{Funding}

This work was supported by the China Postdoctoral Science Foundation [grant numbers 2016M591154, 2017T100058]; the Australian Mathematical Society's Lift-Off Fellowship; and the Australian Academy of Science's AK Head Mathematical Scientists Travelling Fellowship.

\section*{Acknowledgements}

 We owe Kenichi Ohshika an enormous debt of gratitude for patiently sharing his expertise in hyperbolic $3$-manifold theory, without whose help this paper would have been impossible. We would like to thank Hideki Miyachi and Makoto Sakuma for teaching us the ideas of their proof of their extension of the classical Fuchsian McShane identity to the quasifuchsian context.; Greg McShane, Paul Norbury, Athanase Papadopoulos and Ser Peow Tan for helpful discussions and/or feedback. We are also very grateful to the anonymous referees for encouraging words as well as the careful reading of an earlier draft of this paper. Finally, we would also like to thank the IMRN editors and staff for their kind patience.

\appendix

\section{McShane identity for nonorientable cusped hyperbolic surfaces}
\label{sec:appendix}

We now give the derivation for the cuspidal form of Norbury's nonorientable identity from Theorem~2 in\cite{norbury}:

\begin{thm*}[McShane identity for nonorientable surfaces with borders]
Consider a nonorientable hyperbolic surface $N$ with geodesic borders $\beta_1,\ldots,\beta_n$. For
\begin{align}
R(x,y,z)=x-\ln\frac{\cosh\tfrac{y}{2}+\cosh\tfrac{x+z}{2}}{\cosh\tfrac{y}{2}+\cosh\tfrac{x-z}{2}}
\end{align}
and
\begin{align}
D(x,y,z)=R(z,y,z)+R(x,z,y)-x,\; E(x,y,z)=R(x,2z,y)-\tfrac{x}{2}\label{eq:summandsdef}
\end{align}
on a hyperbolic surface with Euler characteristic $\neq1$ the following identity holds:
\begin{align}
\sum_{\alpha,\beta} D(L_1,\ell_{\alpha},\ell_{\beta})
+\sum_{j=2}^n\sum_\gamma R(L_1,L_j,\ell_\gamma)
+\sum_{\mu,\nu}E(L_1,\ell_\nu,\ell_\mu)
=L_1\label{eq:borderedidentity}
\end{align}
where the sums are over simple closed geodesics. The first sum is over pairs of $2$-sided geodesics $\alpha$ and $\beta$ that bound a pair of pants with $\beta_1$, the second sum is over boundary components $\beta_j$, $j=2,\ldots, n$ and $2$-sided geodesics $\gamma$ that bound a pair of pants with $\beta_1$ and $\beta_j$, and the third sum is over $1$-sided geodesics $\mu$ and $2$-sided geodesics $\nu$ that, with $\beta_1$ bound a M\"obius band minus a disk containing $\mu$.
\end{thm*}

As we deform the hyperbolic structure on $N$ so as to approach that of a cusped hyperbolic surface, the lengths of the boundaries $\beta_1,\beta_n$ all tend toward $0$. To obtain a McShane-type identity for a cusped hyperbolic surface, it suffices to divide both sides of \eqref{eq:borderedidentity} by $L_1$ and take the limit as $L_1$ goes to $0$. The are two standard approaches to showing that the resulting term-by-term limit convergences correctly to an identity (as opposed to an inequality with $\leq1$ on the right hand side). The first is to use hyperbolic geometry directly to compute the cuspidal identity and then to compare with the term-by-term limit. The second is to show that \eqref{eq:borderedidentity} divided by $L_1$ is a uniformly convergent series along the path (on the character variety) deforming the hyperbolic structure on $N$ to our desired cuspidal hyperbolic structure. We take this second route, and begin by verifying that Norbury's summands take the desired form in the $L_1\to0$ limit.\medskip

 We first observe the following limit:
\begin{align}
\lim_{L_1\to0}\frac{R(L_1,y,z)}{L_1}=\frac{\cosh\frac{y}{2}+e^{-\frac{z}{2}}}{\cosh\frac{y}{2}+\cosh\frac{z}{2}},
\end{align}
which in turn gives us the following:
\begin{align}
\lim_{L_1,L_j\to0}\tfrac{1}{L_1}R(L_1,L_j,z)&=2(e^{\frac{1}{2}(0+z)}+1)^{-1},\text{ and }\\
\lim_{L_1\to0}\tfrac{1}{L_1}D(L_1,y,z)&=\tfrac{1}{L_1}\left(R(L_1,y,z)+R(L_1,z,y)-L_1\right)\notag\\
&=2(e^{\frac{1}{2}(y+z)}+1)^{-1}.
\end{align}
The above calculations are standard and omitted. Now, since geodesic length functions are continuous over the character variety $\mathrm{Rep}(N)$ of all Fuchsian characters (with with parabolic or hyperbolic boundary holonomy) of $\pi_1(N)$, the existence of the above limits tells us that the functions 
\[
\hat{R}(L_1,L_j,\ell_\gamma):=\tfrac{1}{L_1}R(L_1,L_j,\ell_\gamma)\text{ and }
\hat{D}(L_1,\ell_\alpha,\ell_\beta):=\tfrac{1}{L_1}D(L_1,\ell_\alpha,\ell_\beta)
\]
extend uniquely to continous functions on all of $\mathrm{Rep}(N)$. We employ the same notation to denote their extensions.

\begin{note}
Our notation differs a little from Norbury's here for $\hat{R}(x,y,z)$. To clarify, our $\hat{R}(x,y,z)$ is equal to Norbury's $\hat{D}(x,y,z)+\hat{R}(x,z,y)$.
\end{note}

Already we are beginning to see that summands in the first two series of Norbury's identity are taking the form given in the cuspidal identity. For the final summand, we need to do a little rearranging first. For any embedded $1$-holed M\"obius band $M$ bounded by $\beta_1$ and $\nu$, there are precisely two interior $1$-sided geodesics $\mu$ and $\mu'$. Therefore, each summand in the third term of Norbury's identity arises in a pair $E(L_1,\ell_\nu,\ell_\mu)+E(L_1,\ell_\nu,\ell_{\mu'})$. 
Therefore, we consider the limit
\begin{align}
&\lim_{L_1\to0}\tfrac{1}{L_1}\left(E(L_1,\ell_\nu,\ell_\mu)+E(L_1,\ell_\nu,\ell_{\mu'})\right)\\
&=\lim_{L_1\to0}\tfrac{1}{L_1}(R(L_1,2\ell_\mu,\ell_\nu)+R(L_1,2\ell_{\mu'},\ell_\nu)-L_1)\\
&=1-\frac{\sinh\frac{\ell_\nu}{2}(2\cosh\frac{\ell_\nu}{2}+\cosh\ell_\mu+\cosh\ell_{\mu'})}{(\cosh\frac{\ell_\nu}{2}+\cosh\ell_\mu)(\cosh\frac{\ell_\nu}{2}+\cosh\ell_{\mu'})}.
\end{align}
Using the following trace relation (Equation~{(6)} of\cite{norbury})
\begin{align}
\cosh\tfrac{L_1}{2}+\cosh\tfrac{\ell_\nu}{2}=2\sinh\tfrac{\ell_\mu}{2}\;\sinh\tfrac{\ell_{\mu'}}{2}\text{, with }L_1\text{ set to }0,\label{eq:tracerel}
\end{align}
we can show that
\begin{align}
&(\cosh\tfrac{\ell_\nu}{2}+\cosh\ell_\mu)(\cosh\tfrac{\ell_\nu}{2}+\cosh\ell_|{\mu'})\\
&=(1+\cosh\tfrac{\ell_\nu}{2})(2\cosh\tfrac{\ell_\nu}{2}+\cosh\ell_\mu+\cosh\ell_{\mu'}).
\end{align}
This then tells us that
\begin{align}
\lim_{L_1\to0}\tfrac{1}{L_1}(E(L_1,\ell_\nu,\ell_\mu)+E(L_1,\ell_\nu,\ell_{\mu'}))
=1-\frac{\sinh\frac{\ell_\nu}{2}}{\cosh\frac{\ell_\nu}{2}+1}=2(e^{\frac{1}{2}\ell_\nu}+1)^{-1}.
\end{align}
The fact that this expression should be independent of $\ell_\mu$ and $\ell_{\mu'}$ is, perhaps, somewhat surprising. However, there is a geometric argument for this term which involves cutting up the orientable double cover of $M$ (which is a hyperbolic sphere with two cusps and two geodesic borders of length $\ell_\nu$) along the two ``shortest" ideal geodesics joining its two cusps and regluing each of the two resulting connected components into a pair of pants with two cusps and one boundary of length $\ell_\nu$ (see Figure~\ref{fig:mobiustopants}, but replace geodesic boundaries $\beta_1,\beta_1^A$ and $\beta_1^B$ with cusps as appropriate). We leave this as an exercise for interested readers.\medskip

Equation~\eqref{eq:summandsdef} allows us to break up $E(L_1,\ell_\nu,\ell_\mu)+E(L_1\ell_\nu,\ell_{\mu'})$ even more finely as the following summands:
\begin{align}
E(L_1,\ell_\nu,\ell_\mu)+E(L_1\ell_\nu,\ell_{\mu'})
=& D(L_1,\ell_\nu,2\ell_\mu)+D(L_1,\ell_\nu,2\ell_{\mu'})\label{eq:topDs}\\
& +L_1-R(L_1,\ell_\nu,2\ell_\mu)-R(L_1,\ell_\nu,2\ell_{\mu'}).\label{eq:botRs}
\end{align}
We already know what the limit, as $L_1$ tend to $0$, of \eqref{eq:topDs} divided by $L_1$ is. Therefore, we only need to consider
\begin{align}
&\lim_{L_1\to0}\tfrac{1}{L_1}
\left(L_1-R(L_1,\ell_\nu,2\ell_\mu)-R(L_1,\ell_\nu,2\ell_{\mu'})\right)\\
&=1-\left(
\frac{\cosh\frac{\ell_\nu}{2}+e^{-\ell_\mu}}{\cosh\frac{\ell_\nu}{2}+\cosh\ell_\mu} 
+\frac{\cosh\frac{\ell_\nu}{2}+e^{-\ell_{\mu'}}}{\cosh\frac{\ell_\nu}{2}+\cosh\ell_{\mu'}}
\right).\label{eq:bigfatterm}
\end{align}
As before, the existence of this limit means that
\begin{align}
\hat{E}(L_1,\ell_\mu,\ell_{\mu'}):=\tfrac{1}{L_1}(L_1-R(L_1,\ell_\nu,2\ell_\mu)-R(L_1,\ell_\nu,2\ell_{\mu'}))
\end{align}
extends to a continuous function over all of $\mathrm{Rep}(N)$. We again invoke \eqref{eq:tracerel} to show that
\begin{align}
\cosh\tfrac{\ell_\nu}{2}+e^{-\ell_\mu}&=2\sinh(\ell_{\mu'})\left(-e^{-\frac{\ell_\mu}{2}}+\sinh\tfrac{\ell_{\mu'}}{2}\right),\text{ and }\\
\cosh\tfrac{\ell_\nu}{2}+\cosh\ell_\mu&=2\sinh(\ell_{\mu'})\left(\sinh\tfrac{\ell_\mu}{2}+\sinh\tfrac{\ell_{\mu'}}{2}\right).
\end{align}
Incorporating these two identities into \eqref{eq:bigfatterm}, we obtain:
\begin{align}
&\lim_{L_1\to0}\tfrac{1}{L_1}(L_1-R(L_1,\ell_\nu,2\ell_\mu)-R(L_1,\ell_\nu,2\ell_{\mu'}))\\
&=\frac{e^{-\frac{\ell_{\mu}}{2}}+e^{-\frac{\ell_{\mu'}}{2}}}{\sinh\frac{\ell_\mu}{2}+\sinh\frac{\ell_{\mu'}}{2}}=2(e^{\frac{1}{2}(\ell_\mu+\ell_{\mu'})}-1)^{-1}.
\end{align}
Therefore, we see that the limit of $\frac{1}{L_1}(E(L_1,\ell_\nu,\ell_\mu)+E(L_1,\ell_\nu,\ell_{\mu'}))$ takes the form of three distinct terms, and we get:
\begin{align}
(e^{\frac{1}{2}\ell_\nu}+1)^{-1}
=(e^{\frac{1}{2}(\ell_\nu+2\ell_\mu)}+1)^{-1}+(e^{\frac{1}{2}(\ell_\nu+2\ell_{\mu'})}+1)^{-1}+(e^{\frac{1}{2}(\ell_\mu+\ell_{\mu'})}-1)^{-1}.
\label{eq:triplesum}
\end{align}
There are precisely two pairs of pants embedded on $M$, and they may be obtained from $M$ by cutting along $\mu$ and $\mu'$. One of these pairs of pants has $\beta_1$, $\nu$ and the $2$-sided double-cover of $\mu$ as its boundary, and its corresponding gap term is $(e^{\frac{1}{2}(\ell_\nu+2\ell_\mu)}+1)^{-1}$. The other pair has $\beta_1$, $\nu$ and the $2$-sided double cover of $\mu'$ as its boundary, with corresponding gap term $(e^{\frac{1}{2}(\ell_\nu+2\ell_{\mu'})}+1)^{-1}$. The remaining third term is associated to the M\"obius band $M$. Putting all of this data together with our previous two expressions tells us that the term-by-term limiting identity as $L_1$ tends to $0$ is indeed the one we gave as Norbury's nonorientable cusped surface identity.\medskip

Note that \eqref{eq:triplesum} suggests an alternative statement of the cuspidal case identity:

\begin{thm*}[Alternative cuspidal identity]
Let $\mathcal{S}_1'(N)$ denote the set of $2$-sided geodesics $\gamma$ on $N$ which, along with cusp $p$, bound an embedded $1$-holed M\"obius band and let $\mathcal{S}_2'(N)$ denote the set of unordered pairs of $2$-sided geodesics $\{\alpha,\beta\}$ which, along with cup $p$, bound an embedded pair of pants, which does not lie on an embedded $1$-holed M\"obius band, on $N$. Then,
\begin{align}
\sum_{\{\gamma\}\in\mathcal{S}_1'(N)}\left(e^{\frac{1}{2}\ell_{\gamma}}+1\right)^{-1}
+\sum_{\{\alpha,\beta\}\in\mathcal{S}_2'(N)}\left(e^{\frac{1}{2}(\ell_{\alpha}+\ell_{\beta})}+1\right)^{-1}
=\frac{1}{2}.
\end{align}
\end{thm*}

\begin{note}
\label{note:probability}
It is possible to double the cuspidal version of Norbury's identity and give a probabilistic interpretation of the resulting series in the Fuchsian case. The summand 
\[
2\left(e^{\frac{1}{2}(\ell_{\alpha_2}+\ell_{\beta_2})}+1\right)^{-1},
\]
for the index $\{\alpha_2,\beta_2\}\in\mathcal{S}_2(N)$, is the probability that a geodesic launched from cusp $p$ will self-intersect before intersecting either $\alpha$ or $\beta$. The summand
\[
2\left(e^{\frac{1}{2}(\ell_{\alpha_1}+\ell_{\beta_1})}-1\right)^{-1},
\]
for the index $\{\alpha_1,\beta_1\}\in\mathcal{S}_1(N)$ corresponding to a $1$-holed M\"obius band $M$ containing $\alpha_1,\beta_1$ with boundary $\gamma$, is the probability that a geodesic launched from $p$ will first intersect both $\alpha_1$ and $\beta_1$, and then self-intersect before hitting $\gamma$. For the alternative formulation of the cuspidal case identity, the summand
\[
2\left(e^{\frac{1}{2}\ell_\gamma}+1\right)^{-1},
\]
for the index $\{\gamma\}\in\mathcal{S}_1'(N)$, is the probability that a geodesic launched from $p$ will self-intersect before hitting $\gamma$. Summands for $\mathcal{S}_2'(N)$ have already been discussed as $\mathcal{S}_2'(N)$ is a subset of $\mathcal{S}_2(N)$.
\end{note}

So far, in taking the term-by-term limit of Norbury's bordered surface identity gives us the following inequality:
\begin{align}
\sum_{\{\alpha_1,\beta_1\}\in\mathcal{S}_1(N)}\left(e^{\frac{1}{2}(\ell_{\alpha_1}+\ell_{\beta_1})}-1\right)^{-1}
+\sum_{\{\alpha_2,\beta_2\}\in\mathcal{S}_2(N)}\left(e^{\frac{1}{2}(\ell_{\alpha_2}+\ell_{\beta_2})}+1\right)^{-1}
\leq\frac{1}{2}.
\end{align}
To show that this is in fact an equality, we study the behavior of $\hat{D}(L_1,\ell_\alpha,\ell_\beta)$, $\hat{R}(L_1,L_j,\ell_\gamma)$ and $\hat{E}(L_1,\ell_\mu,\ell_{\mu'})$ as the hyperbolic structure on $N$ deforms to a cuspidal structure. A little algebraic manipulation suffices to show that:
\begin{align}
\hat{D}(x,y,z)
=\frac{2}{x}\left(1+\frac{2\sinh\frac{x}{2}}{e^{\frac{x}{2}}+e^{\frac{y+z}{2}}}\right)
\leq\frac{4\sinh\frac{x}{2}}{x(e^{\frac{x}{2}}+e^{\frac{y+z}{2}})}.
\end{align}
Therefore, when $L_1$ is sufficiently close to $0$, we have 
\[
\hat{D}(L_1,\ell_\alpha,\ell_\beta)<6e^{-\frac{1}{2}(\ell_\alpha+\ell_\beta)}.
\] 
For $\hat{R}(L_1,L_j,\ell_\gamma)$, we utilize an alternative expression (see, e.g.: Equation~(1.7) of\cite{twzcone}):
\begin{align}
\hat{R}(x,y,z)
&=2\tanh^{-1}\left(\frac{\sinh\frac{x}{2}\sinh\frac{y}{2}}{\cosh\frac{z}{2}+\cosh\frac{x+y}{2}}\right)+\hat{D}(x,y,z)\\
&=\frac{1}{x}\log\left(1+\frac{2\sinh\frac{x}{2}\sinh\frac{y}{2}}{\cosh\frac{z}{2}+\cosh\frac{x}{2}\cosh\frac{y}{2}}\right)+\hat{D}(x,y,z)\\
&\leq\frac{2\sinh\frac{x}{2}\sinh\frac{y}{2}}{\cosh\frac{z}{2}+\cosh\frac{x}{2}\cosh\frac{y}{2}}+\hat{D}(x,y,z).
\end{align}
Therefore, when $x$ and $y$ are sufficiently close to $0$, we have 
\[
\hat{R}(L_1,L_j,\ell_\gamma)<3e^{-\frac{1}{2}\ell_\gamma}+3e^{-\frac{1}{2}(L_j+\ell_\gamma)}<6e^{-\frac{1}{2}\ell_\gamma}.
\]
For $\hat{E}(L_1,\ell_\mu,\ell_{\mu'})$, we employ a small geometric argument. Firstly, we know by construction that $\hat{E}(L_1,\ell_\mu,\ell_{\mu'})<\tfrac{1}{L_1}(E(L_1,\ell_\nu,\ell_\mu)+E(L_1\ell_\nu,\ell_{\mu'}))$ and we bound this larger expression instead. On the $1$-holed M\"obius band $M$ bounded by $\beta_1$ and $\nu$, there is a unique simple $1$-sided orthogeodesic $\sigma$ with both endpoints based on $\beta_1$. Cutting $M$ along $\sigma$ results in an annulus, and we may reglue the two sides of this annulus along $\sigma$ in an orientation preserving way (see Figure~\ref{fig:mobiustopants}) so as to obtain a pair of pants with boundaries $\nu$, $\beta_1^A$ and $\beta_1^B$ of respective lengths $\ell_\nu$, $L_1^A$ and $L_1^B$ such that $L_1^A+L_1^B=L_1$.\medskip

\begin{figure}[h!]
\centering
\includegraphics[width=0.8\textwidth]{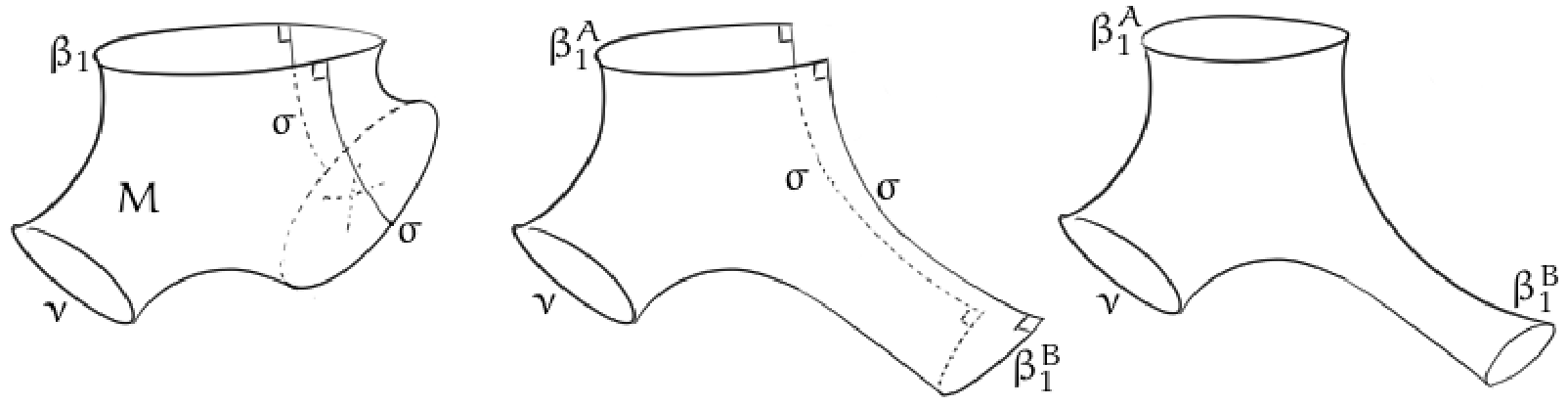}
\caption{A M\"obius band $M$ (left) cut along $\sigma$ (center) and reglued to form a pair of pants (right).}
\label{fig:mobiustopants}
\end{figure}

In particular, we can see from Figure~\ref{fig:doublemobius}, which lifts $M$ to its orientable double cover $dM$, that the gaps corresponding to $E(L_1,\ell_\nu,\ell_\mu)+E(L_1,\ell_\nu,\ell_{\mu'})$ actually equal the sum of the two gaps on $\beta_1^A$ and $\beta_1^B$ with total measure
\[
R(L_1^A,L_1^B,\ell_\nu)+R(L_1^B,L_1^A,\ell_\nu)
\]
because the cutting and regluing procedure does not affect the positions of the four (red) geodesics spiraling to $\nu$. Thus, we see that when $L_1$ (hence $L_1^A$ and $L_1^B$) is sufficiently close to $0$,
\begin{align}
\tfrac{1}{L_1}\left(E(L_1,\ell_\nu,\ell_\mu)+E(L_1,\ell_\nu,\ell_{\mu'})\right)
&=\tfrac{1}{L_1}\left(\hat{R}(L_1^A,L_1^B,\ell_\nu)L_1^A+\hat{R}(L_1^B,L_1^A,\ell_\nu)L_1^B\right)\notag\\
&<\tfrac{1}{L_1}(6e^{-\frac{1}{2}\ell_\nu}L_1^A+6e^{-\frac{1}{2}\ell_\nu}L_1^B)=6e^{-\frac{1}{2}\ell_\nu}.
\end{align}

\begin{figure}[h!]
\centering
\includegraphics[width=0.8\textwidth]{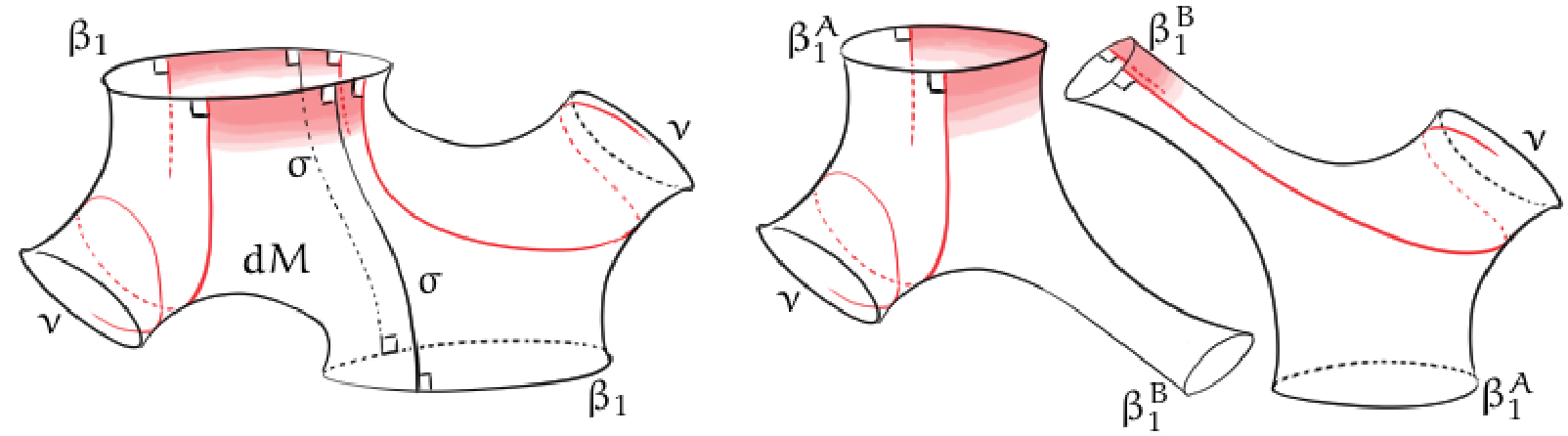}
\caption{The red shaded region on the left diagram has length $E(L_1,\ell_\nu,\ell_\mu)+E(L_1,\ell_\nu,\ell_{\mu'})$ along $\beta_1$; the two red shaded regions on the right half respectively have lengths $R(L_1^A,L_1^B,\ell_\nu)$ and $R(L_1^B,L_1^A,\ell_\nu)$.}
\label{fig:doublemobius}
\end{figure}

At this point, we may invoke the same argument as used in the proof of Proposition~\ref{thm:holomorphicity} to obtain a similar polynomial divided by exponential type expression as \eqref{eq:uniformbound} for the tail of Norbury's identity (upon appropriate rearrangement of the series). This ensures the uniform convergence of the bordered case identity as the hyperbolic structure on $N$ deforms to a cusped structure, thus allowing us to conclude that the term-by-term limit is in fact an equality.

\begin{note}
The above arguments obviously apply when the underlying surface is orientable, thus furnishing the nitty-gritty details for the proof of Mirzakhani's Corollary~4.3.
\end{note}

\begin{note}
Our uniform convergence arguments also apply when an interior simple closed geodesic deforms to a cusp. Therefore, starting with a McShane identity for a surface with greater topological complexity, we may take these limits to derive identities for surfaces with lower complexity simply by taking term-by-term limits. In particular, simple geodesics which intersect the shrinking geodesic(s) must tend to length $\infty$ and summands expressing their lengths therefore tend to $0$ and are excluded from the identity. This was previously noted in the special case when a pair of simple closed geodesics $\alpha,\beta$, which bound a pair of pants with $\beta_1$, deform to cusps (Example~2.2~(2) of\cite{amsbundles}). 
\end{note}

\end{document}